\documentclass[12pt,reqno]{amsart}
\usepackage{amssymb, amsmath, amsthm}
\usepackage[latin1]{inputenc}
\usepackage{graphics,graphicx}
\graphicspath{{./figs/}}
\usepackage[final]{hyperref}
\usepackage{color}
\usepackage{algorithm}
\usepackage{algorithmicx}
\usepackage{algpseudocode}
\usepackage{ulem}
\newcommand\ds\displaystyle
\usepackage[a4paper, centering]{geometry}

\usepackage{color}
\usepackage{graphicx}
\graphicspath{{./pics/}}
\newtheorem{theorem}{Theorem}[section]
\newtheorem{proposition}[theorem]{Proposition}
\newtheorem{lemma}[theorem]{Lemma}

\theoremstyle{definition}

\theoremstyle{remark}
\newtheorem{remark}[theorem]{Remark}

\def\R{\mathbb{R}}
\def\RN{{\R^N}}

\def\HN{{\mathbb{H}^N}}

\definecolor{verde}{RGB}{20,150,100}

\newcommand{\Om}{\Omega}
\newcommand{\om}{\omega}
\newcommand{\Hyp}{{\mathbb H}}
\newcommand{\Sph}{{\mathbb S}}
	
\DeclareMathOperator{\sn}{{sn}}

\DeclareMathOperator{\sign}{sign}

\newcommand{\e}{\varepsilon}

\newcommand{\lb}{\lambda}
\newcommand{\MC}{{\mathcal C}_{plate}}
\newcommand{\MCM}{{\mathcal C_{mbr}}}

\newcommand{\sm}{\setminus}

\newcommand{\sq}{\subseteq}
\newcommand{\ov}{\overline}

\newcommand{\Li}{\operatorname{Li}}

\usepackage{MnSymbol}
\newcommand{\mystar}{\filledstar}
\newcommand{\1}{\mathbf{1}}

\begin{document}

\title[Fourth order Saint-Venant inequalities]{Fourth order Saint-Venant inequalities: maximizing compliance and mean deflection among clamped plates}

\author[M. Ashbaugh]
{Mark Ashbaugh }
\address[Mark Ashbaugh]{University of Missouri, Columbia, USA}
\email[M. Ashbaugh]{ashbaughm@missouri.edu}
\author[D. Bucur]
{Dorin Bucur}
\address[Dorin Bucur]{ Universit\'e Savoie Mont Blanc, CNRS, LAMA,
73000 Chamb\'ery, France
}
\email[D. Bucur]{dorin.bucur@univ-savoie.fr}
\author[R. S. Laugesen]
{Richard S. Laugesen}
\address[Richard S. Laugesen]{ University of Illinois,
Urbana IL 61801, USA
}
\email[R. S. Laugesen]{Laugesen@illinois.edu}

\author[R. Leylekian]
{Rom\'eo Leylekian}
\address[Rom\'eo Leylekian]{ Grupo de F\'{\i}sica Matem\'{a}tica, Instituto Superior T\'{e}cnico, Universidade de Lisboa, Av. Rovisco Pais, 1049-001 Lisboa, Portugal
}
\email[R. Leylekian]{romeo.leylekian@univ-amu.fr}


\date{\today}
\maketitle

\begin{abstract}   
We prove a fourth order analogue of the Saint-Venant inequality: the mean deflection of a clamped plate under uniform transverse load is maximal for the ball, among plates of prescribed volume in any dimension of space. The method works in Euclidean space, hyperbolic space, and the sphere. 

Similar results for clamped plates under small compression and for the compliance under non-uniform loads are proved to hold in two dimensional Euclidean space, with the higher dimensional and  curved cases  of those problems left open.
 \end{abstract}

\section{\bf Introduction}

Consider a clamped plate $\Om$ under transverse load $\rho$, where  $\Om \sq \RN$, $N \ge 1$, is open and has finite measure and $\rho : \Om \to [-1,1]$. Writing $\sigma$ for the lateral compression parameter, the equation is  
\begin{equation}
\label{as01}
\begin{cases}
\Delta^{\! 2} u +\sigma^2 \Delta u=\rho&\text{in }\Om, \\
u=\frac{\partial u}{\partial n} =0&\text{on }\partial\Om .
\end{cases}
\end{equation}
When $\sigma^2$ is less than the first buckling eigenvalue $\Lambda_1(\Om)$ of the clamped plate, the equation is well posed in a weak variational sense (Section \ref{sec:s2}) and has a unique solution we denote by $u_{\Om, \rho}$. We aim to maximize the {\it compliance} of the plate, 
 \[
 {\MC} (\Om, \rho) = \ds \int_\Om \rho u_{\Om, \rho} \, dV.
 \]
 Notice that when $\rho \equiv 1$, the compliance 
 \[
 {\MC} (\Om, 1)= \ds  \int_\Om u_{\Om, 1} \, dV
 \]
coincides with  the {\it mean} deflection of the plate. The goal is to maximize compliance while varying both $\Om$ and $\rho$: to find 
\begin{equation}
\label{as03}
\max \MC(\Om, \rho) \end{equation}
when the volume $V(\Om)=m$ is prescribed and the load $\rho$ runs over $L^\infty (\Om, [-1,1])$.

\subsection*{Second order analogues} This question is reminiscent of Saint-Venant's inequality for the deflection of a loaded membrane (or alternatively, torsional rigidity of a beam having that simply connected cross-section). But the plate problem is more challenging, as we now explain. Let $w_{\Om,\rho}\in H^1_0(\Om)$ be the weak solution of 
\[
\begin{cases}
-\Delta w =\rho&\text{in }\Om, \\
w=0&\text{on }\partial\Om .
\end{cases}
\]
The Saint-Venant inequality states that mean deflection under constant load is largest for the ball:
\[
\int_\Om w_{\Om,1}  \, dV \le \int_{\Om^*}  w_{\Om^*,1} \, dV,
\]
where $\Om^*$ is a ball in $\RN$ with the same volume as $\Om$.

In this membrane case, if the constant load $1$ is replaced by any other $\rho\in L^\infty (\Om, [-1,1])$, then the maximum principle yields that $|w_{\Om, \rho}| \le w_{\Om, |\rho|} \le  w_{\Om, 1}$. Writing $\MCM(\Om, \rho)=\int _\Om w_{\Om, \rho}\rho \, dV$ for the compliance of the membrane, we deduce 
\[
\MCM(\Om, \rho)\le \MCM(\Om, |\rho|) \leq \MCM(\Om, 1) .
\]
Combining with Saint-Venant, one see that the maximal compliance 
\[
\max\{\MCM(\Om, \rho) : V(\Om)=m \text{ and } \rho \in L^\infty (\Om, [-1,1])\}
\]
is attained on the ball of volume $m$ with $\rho \equiv 1$. 

The analogous reasoning fails for the plate problem because the maximum principle does not hold for the fourth order bilaplacian. It could perhaps happen that for some couple $(\Om, \rho)$ we have $\MC(\Om, \rho)>\MC(\Om, 1)$.  Consequently, when addressing problem \eqref{as03} it is not enough to restrict attention to the constant load $\rho \equiv 1$, because the problem 
\begin{equation}
\label{as03.1}
 \max\{\MC(\Om, 1) : V(\Om)=m\}
 \end{equation}
 is a priori different from \eqref{as03} and might have a different solution. However, we conjecture that both problems \eqref{as03} and \eqref{as03.1} are maximized by the ball with $\rho \equiv 1$. 
 
\subsection*{Fourth order results}
Our plate results are as follows. 

\subsubsection*{Constant load, no compression, all dimensions,   constant  curvatures --- Theorem \ref{as10general}} For constant load $\rho \equiv 1$ and zero compression $\sigma=0$, we prove that the compliance $\MC (\Om,1)$ is maximal for a geodesic ball, in any dimension of Euclidean or hyperbolic space or the sphere. An explicit formula for this maximal compliance is found in Proposition \ref{pr:maxvalue} for $N$-dimensional Euclidean balls and for $2$-dimensional hyperbolic and spherical balls. 

\subsubsection*{Constant load, small compression, Euclidean plane --- Theorem \ref{as13}} For open sets in the Euclidean plane under constant load $\rho \equiv 1$, the compliance is maximal for the ball provided the compression parameter is sufficiently close to $0$.

\subsubsection*{Variable load, no compression, Euclidean plane --- Theorem \ref{th:total}} In the plane, the compliance problem \eqref{as03} is solved in full generality when $\sigma=0$. 

Open problems are raised in Section \ref{sec:compliance} for mean deflection and total deflection, and we close the paper in Section \ref{sec:alternative} with remarks on possible alternative methods of proof. 

 \subsection*{Methods and related literature}
The Saint-Venant inequality can be proved straightforwardly by symmetric decreasing rearrangement (Schwarz symmetrization), as for Faber--Krahn and many other isoperimetric inequalities associated with the Laplacian. Inequalities for fourth order operators are in general much more difficult. 
 
The isoperimetric inequality for the first eigenvalue of a clamped plate ($\Delta^2 u = \Gamma u$) has been proved in two and three dimensions by Nadirashvili \cite{Nad95} and Ashbaugh-Benguria \cite{AshBen95} respectively, building on the approach of Talenti \cite{Tal81}. For the hyperbolic and spherical situations, see Krist\'{a}ly \cite{Kristaly2020,Kristaly2022}. These proofs are all based on the introduction of a two-ball auxiliary problem via Talenti's elliptic comparison theorem. The limitation to dimensions 2 and 3 arises because in dimensions 4 and higher, the minimizer of the auxiliary problem is not a single ball but instead consists of two equal-sized balls, by Ashbaugh and Laugesen \cite{AL96,ABL97}. For similar reasons, the buckling load minimization problem for clamped plates remains open even in dimension $2$. 

The recent work of Leylekian \cite{L24a,L24b} introduces a new approach to these plate eigenvalue problems, but does not resolve the conjectures. 
 
For the variable load result in Theorem \ref{th:total}, the auxiliary problem differs critically from the constant load case and inherits all the expected difficulties of the fourth order PDE. To prove the theorem we proceed explicitly and solve the auxiliary problem exactly in all dimensions. But only in dimensions $1$ and $2$ does the auxiliary problem yield a single ball as the maximizer. Thus, just as for the clamped plate eigenvalue problem, one runs up against the technical limitations of the two-ball method. We also provide an alternative proof relying on a new, direct comparison result that extends Talenti's theorem to sign changing solutions. However,   this approach   still works only in dimension $2$  (see Theorem \ref{thm:talenti-signe}).

\section{\bf Constant load results ($\rho \equiv 1$): zero compression ($\sigma=0$) in Euclidean, spherical and hyperbolic spaces, and small compression in the plane}\label{sec:s2}

Consider the three constant curvature manifolds, Euclidean space $\RN$, the round unit sphere $\Sph^N$ with curvature $+1$, and $N$-dimensional hyperbolic space $\HN$ with sectional curvature $-1$. We assume $N \geq 1$ for the Euclidean space and $N \geq 2$ for the sphere and hyperbolic space. 

In each situation, denote the space by $M$, the metric by $\mathfrak{g}$, the Laplace--Beltrami operator by $\Delta_\mathfrak{g}$, and the volume element by $dV=dV_\mathfrak{g}$. Here $\Delta_\mathfrak{g}$ is the ``geometer's Laplacian'', a positive operator. To maintain consistency with the Euclidean case, we write 
\[
\Delta=-\Delta_\mathfrak{g}
\]
for the ``analyst's Laplacian'', and use that operator consistently. Write $\nabla$ for the gradient operator and $|\nabla v|=\mathfrak{g}(\nabla v,\nabla v)^{1/2}$ for the length of the gradient vector with respect to the metric.

\subsection{Mean deflection and minimal energy} Let $\Om \subset M$ be a nonempty open set with finite measure,  so that $H^2_0(\Om)$ imbeds compactly into $L^2(\Om)$. On the sphere we further require $1 \notin H^2_0(\Om)$, which ensures $\Om$ cannot be the whole sphere.  The first buckling eigenvalue of the clamped plate  ($\Delta^2 u + \Lambda \Delta u = 0$)  is 
\[
\Lambda_1(\Om)=\inf \big\{ \int_\Om |\Delta v|^2 \, dV \Big/ \int_\Om |\nabla v|^2 \, dV : 0 \not\equiv v \in H^2_0(\Om) \big\} > 0.
\]
 
Given a parameter $\sigma \ge 0$, consider the clamped plate equation on $\Omega$ under compression and constant load:
\begin{equation}
\label{as01general}
\begin{cases}
\Delta^{\! 2} u_\Om+\sigma^2 \Delta u_\Om=1&\hbox{in }\Om,\\
u_\Om=\frac{\partial u_\Om}{\partial n} =0&\hbox{on }\partial\Om.
\end{cases}
\end{equation}
Provided $\sigma ^2 < \Lambda_1(\Om)$,   the operator $\Delta^{\! 2} +\sigma^2 \Delta$ is associated to a coercive quadratic form  and the equation is well posed in the weak variational sense, where the weak form says
\[
\int_\Om \Delta u_\Om \, \Delta v \, dV- \sigma ^2 \int_\Om \nabla u_\Om \cdot \nabla v  \, dV = \int_\Om v \, dV, \qquad v \in H^2_0(\Om) .
\]
If $\Om$ has smooth boundary then weak solutions are classical smooth solutions. 

The existence of a solution comes from the minimization of the following energy
\[
E(\Om, \sigma)= \min_{v \in H^2_0(\Om)} \frac{1}{2} \int_\Om |\Delta v|^2 \, dV -\frac{\sigma^2}{2} \int_\Om |\nabla v|^2 \, dV -\int_\Om v\, dV .
\]
Note that if $\sigma^2 >0$ then the energy is not formally convex. However, uniqueness of the solution is a consequence of linearity of the PDE and the restriction $\sigma^2 <\Lambda_1(\Om)$.

Any minimizer $u_\Om$ satisfies equation \eqref{as01general} and so by choosing $v=u_\Om$ in the weak formulation we deduce (using our assumption $\sigma ^2 < \Lambda_1(\Om)$) that the mean deflection of the solution is positive and equals $-2$ times the energy:
\begin{align}
\int_\Om u_\Om \, dV
& = \int_\Om |\Delta u_\Om|^2 \, dV - \sigma^2 \int_\Om |\nabla u_\Om|^2 \, dV > 0 \label{deflectionpositive} \\
& =  - 2 E(\Om, \sigma) . \notag 
\end{align}
In particular, the energy is negative:
 \[
 E(\Om, \sigma) < 0.
 \]
 
Further, the energy is  domain monotonic, with 
\[
\Omega_1 \subseteq \Omega_2 \quad \Longrightarrow \quad E(\Omega_1, \sigma) \ge E(\Omega_2, \sigma).
\]
This inequality is strict if $\Om_1$ differs from $\Om_2$ by a {\it substantial} set, for example a set of positive measure. For if $E(\Omega_1, \sigma) = E(\Omega_2, \sigma)$ then after extending the solution on $\Om_1$ by $0$, it is also an energy minimizer and hence a solution of \eqref{as01general} on $\Om_2$ and so its derivatives vanish almost everywhere on $\Om_2\sm \Om_1$, contradicting \eqref{as01general} on that set.

\subsection{Zero compression}
Our results are strongest when the plate is not subjected to lateral compression, that is, when $\sigma=0$. The partial differential equation \eqref{as01general} then simplifies to  
\begin{equation}
\label{as01zero}
\begin{cases}
\Delta^{\! 2} u_\Om=1&\hbox{in }\Om,\\
u_\Om=\frac{\partial u_\Om}{\partial n} =0&\hbox{on }\partial\Om,
\end{cases}
\end{equation}
and the energy is 
\[
E(\Om, 0)= \min_{v \in H^2_0(\Om)} \frac{1}{2} \int_\Om |\Delta v|^2 \, dV -\int_\Om v\, dV .
\]
\begin{theorem}[Ball minimizes the energy when compression is absent] \label{as10general}
Let $\Om \subset M$ be a nonempty open set with finite measure, and in the spherical case assume $1 \notin H^2_0(\Om)$. If $\Om^*$ is a geodesic ball with the same volume as $\Omega$ then the unique solution $u_\Om$ of \eqref{as01zero} satisfies
\begin{equation}
\label{as04general}
\int_\Om u_\Om \, dV \le \int_{\Om^*}  u_{\Om^*} \, dV ,
\end{equation}
or equivalently
\[
E(\Om,0) \geq E(\Om^*,0) .
\]
Equality occurs if and only if $\Om$ coincides with a geodesic ball, up to a set of $H^2$-capacity zero.
\end{theorem}

We refer to \cite[Sections 3.3 and 3.8]{henrot-pierre} for an introduction to $H^1$- and $H^2$-capacities. One may regard $H^2$-capacity $\text{cap}_{H^2}$ as an outer measure that is finer than $H^1$-capacity $\text{cap}_{H^1}$, which is itself finer than Lebesgue measure, in the sense that  
\[
\text{cap}_{H^2}(E)=0 \quad \Longrightarrow \quad \text{cap}_{H^1}(E)=0 \quad \Longrightarrow \quad V(E)=0 
\]
for any set $E\subseteq M$.

\subsection{Extremal energies under zero compression} For the ball $\Omega^*$, the deflection function and energy appearing in Theorem \ref{as10general} can be computed explicitly in all Euclidean dimensions and also for the $2$-dimensional spherical and hyperbolic situations. 
 
 \smallskip
 \noindent \textsc{$N$-dimensional Euclidean space $M=\RN$ for $N \geq 1$}: the function satisfying the fourth order equation and boundary conditions in \eqref{as01general} with $\sigma=0$ is the radial function on the ball $\Omega^*$ of radius $R$ given by 
\begin{align} 
u_{\Om^*}(r) 
& = \frac{R^4}{8N(N+2)} \big( 2(1 - r^2/R^2) - (1 - r^4/R^4) \big) \notag \\
& = \frac{R^4}{8N(N+2)} (1 - r^2/R^2)^2 , \qquad r \in [0,R] .\label{euclidexplicit}
\end{align}
To verify this solution, one computes directly using that the Laplacian of a radial function is $r^{1-N}(r^{N-1}f^\prime(r))^\prime$.

 \smallskip
 \noindent \textsc{$2$-dimensional sphere $M=\Sph^2$}: the corresponding solution on a spherical cap $\Omega^*$ of radius $0 < R < \pi$ is
\[
\begin{split}
u_{\Om^*}(r)
& = 2 (\log \cos(r/2))^2 - 2 (\log \cos(R/2))^2 \\
& \quad + \Li_2(-\tan^2(r/2)) - \Li_2(-\tan^2(R/2)) + c \log \frac{\cos(r/2)}{\cos(R/2)} 
\end{split}
\]
where $c=4 \cot^2(R/2) \log \cos(R/2)$ and $\Li_2(z)=\sum_{k=1}^\infty z^k/k^2$ is the polylogarithm function. The solution is readily checked, since the spherical Laplacian of a function that depends only on the geodesic distance $r$ from the north pole on $\Sph^2$ is $(\sin r)^{-1}\big( (\sin r) f^\prime(r) \big)^\prime$.

 \smallskip
 \noindent \textsc{$2$-dimensional hyperbolic space $M=\Hyp^2$}: the solution on a geodesic disk $\Omega^*$ of radius $R>0$ is
\[
\begin{split}
u_{\Om^*}(r)
& = 2 (\log \cosh(r/2))^2 - 2 (\log \cosh(R/2))^2 \\
& \quad + \Li_2(\tanh^2(r/2)) - \Li_2(\tanh^2(R/2)) + c \log \frac{\cosh(r/2)}{\cosh(R/2)} 
\end{split}
\]
where $c=-4 \coth^2(R/2) \log \cosh(R/2)$, since the hyperbolic Laplacian of a function depending only on the geodesic distance  from the origin in the hyperbolic plane $\Hyp^2$ is $(\sinh r)^{-1}\big( (\sinh r) f^\prime(r) \big)^\prime$.
 
 \medskip

We have not succeeded in evaluating $u_{\Om^*}(r)$ on spherical or hyperbolic balls in dimension $N \geq 3$.

Integrating the above formulas for the deflection function yields the following explicit expressions for the mean deflection. 
\begin{proposition}[Maximal mean deflection under zero compression] \label{pr:maxvalue}
In the cases treated above, the maximum value of the mean deflection in \eqref{as04general} for the ball $\Om^*$ of radius $R$ is   
\[
\begin{split}
& \int_{\Om^*}  u_{\Om^*} \, dV \\
& = \begin{cases}
|\Sph^{N-1}| \, \frac{R^{N+4}}{N^2(N+2)^2 (N+4)} , & M=\RN, N \geq 1 ,\\
4\pi \big( \sin^2(R/2) - 4 \cot^2(R/2) (\log \cos(R/2))^2 \big) , & \text{for $M=\Sph^2$,} \\
4\pi \big( \sinh^2(R/2) - 4 \coth^2(R/2) (\log \cosh(R/2))^2 \big) , & \text{for $M=\Hyp^2$.} 
\end{cases}
\end{split}
\]
\end{proposition}
Multiplying these deflection values by $-1/2$ yields the the minimal energy values. 

\subsection{Small positive compression}
When the lateral compression parameter $\sigma$ is sufficiently small, we are able to prove a fourth order St.\ Venant theorem in the planar Euclidean case. 
\begin{theorem}[Disk minimizes the energy, under compression] \label{as13}
Let $N = 2$.  There exists $\sigma_2>0$ such that if $\sigma \in [0, \sigma_2]$ and $\Om \subset \R^2$ is an open set with area equal to $\pi$, then every energy minimizing solution $u_\Om$ of \eqref{as01general}    satisfies
\[
\int_\Om u_\Om \, dV \le \int_{\Om^*}  u_{\Om^*} \, dV ,
\]
or equivalently $E(\Om,\sigma) \geq E(\Om^*,\sigma)$. 
Equality occurs if and only if \/$\Om$ coincides with a unit disk.
\end{theorem}

\begin{remark}
The compression threshold $\sigma_2$ will generally change when the volume constraint is changed because the scaling law for the energy, which says $E(\alpha\Omega,\sigma)=\alpha^{N+4}E(\Omega,\alpha\sigma)$ for $\alpha>0$, scales the compression parameter $\sigma$ too.  
\end{remark}

\begin{remark}
Unlike in Theorem \ref{as10general}, it is not necessary for Theorem \ref{as13} to say equality holds only up to sets of zero $H^2$-capacity. The difference is that when $N=2$, even a single point has positive $H^2$-capacity.
\end{remark}

\section{\bf Proof of Theorem \ref{as10general} --- constant load, zero compression}
\label{sec:as10generalproof}

\subsection*{The case $N=1$} If the open set $\Om \subset \R$ consists of finitely many open intervals then we may translate them to be adjacent, so that the union of their closures is exactly the interval $\Om^*$ together with its two endpoints. Similarly, if $\Om$ consists of infinitely many open intervals then we may reorder and translate them to be adjacent, say, ordered from left to right in decreasing order of size. This ordering is possible since for each $\e>0$, only finitely many intervals have length greater than $\e$. Reshuffling $\Omega$ in this fashion clearly does not change the energy $E(\Om,0)$ and ensures $\Om\subseteq \Om^*$.

The inclusion $H^2_0(\Om) \subseteq H^2_0(\Om^*)$ then implies the inequality $E(\Om, 0) \geq E(\Om^*, 0)$, by using $u_\Om$ as a trial function for $E(\Om^*, 0)$. If equality holds then $u_\Om = u_{\Om^*}$ by uniqueness of the minimizer, and since $u_{\Om^*}$ is positive (by \eqref{euclidexplicit} with $N=1$) it follows that $u_\Om$ vanishes only at the endpoints of $\Om^*$, and so $\Om$ must consist of a single interval. 

\subsection*{The case $N \geq 2$} The proof involves Talenti's comparison theorem and a two-ball auxiliary problem, as inspired by the solution of Rayleigh's conjecture for the fundamental tone of the vibrating clamped plate by Talenti \cite{Tal81}, Nadirashvili \cite{Nad95} and Ashbaugh--Benguria \cite{AshBen95}. Unlike in those earlier works, in this compliance question the auxiliary problem that we construct is not symmetric. This symmetry breaking takes advantage, essentially, of the fact that one does not expect large regions where the solution $u_\Om$ is negative. 

\subsubsection*{\sc Step 1 --- balls} Write $B_r \subset M$ for the ball of geodesic radius $r>0$ that is centered at the origin in Euclidean space, or centered at the north pole on the sphere, or centered in hyperbolic space at some point that from now on we regard as the origin.  

Write $B_R=\Om^*$ for the ball with which we wish to compare $\Omega$. The solution $u=u_\Om$ of \eqref{as01zero} is real analytic in a suitable coordinate chart (which in the spherical case we may take to be a global, stereographic chart) since the operator $\Delta^{\! 2}$ has real analytic coefficients and hence is analytically hypoelliptic \cite[Theorem 3.1]{Shimakura1992}. 

Decompose $u_\Om$ into positive and negative parts as $u_\Om= u_+-u_-$, denote the positive and negative nodal sets as $\Om_+=\{u_+>0\}$ and $\Om_-=\{u_->0\}$, and write $a,b \geq 0$ for the geodesic radii of the balls having the same volumes:
\[
V(\Om_+) = V(B_a), \qquad V(\Om_-)=V(B_b).
\]
Write 
\[
V=V(\Om_+)+V(\Om_-)=V(B_a)+V(B_b) .
\]
The set $\{ u = 0 \}$ has zero volume, because if $u$ vanished on a set of positive measure in some component of $\Om$ then  analyticity would force $u$ to vanish on the whole component  \cite{mityagin}, contradicting equation \eqref{as01zero}. Hence $V=V(\Om)=V(B_R)$. Thus in what follows, $a$ and $b$ are related always by the volume constraint $V(B_a)+V(B_b)=V(B_R)$. 

The critical set $\{ |\nabla u|=0 \}$ also has zero volume, since if $\nabla u$ vanished on a set of positive measure in some component of $\Om$ then analyticity would force it to vanish on the whole component, making $u$ constant there and again contradicting \eqref{as01zero}. 

\subsubsection*{\sc Step 2 --- elliptic rearrangement} Write $-F=(-\Delta u_+)^*$ on $B_a$ for the symmetric decreasing rearrangement of $-\Delta u_+$ restricted to $\Omega_+$, and write $-G=(-\Delta u_-)^*$ on $B_b$ for the symmetric decreasing rearrangement of $-\Delta u_-$ restricted to $\Omega_-$. The definition of symmetric decreasing rearrangement ensures equimeasurability even for signed functions; see for instance \cite{Kesavan2006}. In particular, $F$ and $G$ might change sign. 

This $F$ and $G$ are square-integrable, and are finite-valued and continuous (except perhaps at the centers of $B_a$ and $B_b$) since $\Delta u_\pm$ is continuous. Take radial functions $f \in H^1_0 \cap H^2$ and $g \in H^1_0 \cap H^2(B_b)$ that solve the Poisson equations  
\[
\Delta f = F \quad \text{ on } B_a , \qquad \Delta g = G \quad \text{ on } B_b ,
\]
noting that $f$ and $g$ are $C^2$-smooth (except perhaps at the centers of their respective balls). From a Talenti-style elliptic comparison result, Theorem \ref{talenti} in the Appendix, we know  $f$ and $g$ are symmetric decreasing as a function of the geodesic distance $r$ from the origin with $0 \leq u_+^* \le f$ and $0 \leq u_-^* \le g$. The theorem yields that $ \int_{\Om_+} u_+ \, dV \leq \int_{B_a} f \, dV$, and that if equality holds and $f$ is not identically zero then $\Om^+$ is a ball up to a set of measure zero. Hence  
\begin{align}
E(\Om, 0)
& = \frac{1}{2} \int_{\Om_+} |\Delta u_+|^2 \, dV  - \int_{\Om_+} u_+ \, dV +  \frac{1}{2} \int_{\Om_-} |\Delta u_-|^2 \, dV +\int_{\Om_-} u_- \, dV \notag \\
& \ge\frac{1}{2} \int_{B_a} |\Delta f|^2 \, dV - \int_{B_a} f \, dV + \frac{1}{2} \int_{B_b} |\Delta g|^2 \, dV , \label{eq:asymmetricgeneral}
\end{align}
where we simply dropped the nonnegative term $\int_{\Om_-} u_- \, dV $, thus ``breaking the symmetry'' of the functional with respect to $f$ and $g$. Note that $f$ and $g$ are constrained (or linked) by the requirement that $\int_{B_a} \Delta f \, dV = \int_{B_b} \Delta g \, dV$, as we now explain. Indeed, $\int_\Om \Delta u_\Om \, dV = \int_{\partial \Om} \frac{\partial u_\Om}{\partial n} \, dS = 0$ since $u_\Om \in H^2_0(\Om)$, and so $\int_{\Om_+} \Delta u_+ \, dV - \int_{\Om_-} \Delta u_- \, dV = 0$. Hence 
\begin{align*}
\int_{B_a} \Delta f \, dV 
  = \int_{B_a} F \, dV 
&  = \int_{\Om_+} \Delta u_+ \, dV \\
& = \int_{\Om_-} \Delta u_- \, dV = \int_{B_b} G \, dV = \int_{B_b} \Delta g \, dV ,
\end{align*} 
as claimed. These integrals are all nonpositive, since $\int_{B_a} \Delta f \, dV$ equals the integral of the normal derivative of $f$ over the boundary of the ball, and that derivative is $\leq 0$ because $f$ is radially decreasing. 

\subsubsection*{\sc Step 3 --- two-ball problem} Inspired by expression \eqref{eq:asymmetricgeneral}, we introduce a two-ball auxiliary problem for $a \in [0,R]$: 
\begin{equation}
\label{twoballenergygeneral}
{\mathcal E}(a) = \min \left( \frac{1}{2} \int_{B_a} |\Delta f|^2 \, dV - \int_{B_a} f \, dV + \frac{1}{2} \int_{B_b} |\Delta g|^2 \, dV \right) 
\end{equation}
where the minimum is taken over ``admissible'' pairs $(f,g)$, meaning: 
\begin{align}
f \in H^1_0 \cap H^2(B_a), & \qquad g \in H^1_0 \cap H^2(B_b) ,  \label{constraintgeneral.11} \\
\int_{B_a} \Delta f \, dV & = \int_{B_b} \Delta g \, dV \leq 0 . \label{constraintgeneral}
\end{align}
When $a=0$, one omits $f$ from the energy functional in \eqref{twoballenergygeneral} and in that case the minimizer is obviously $g \equiv 0$ with ${\mathcal E}(0)=0$. When $a=R$, one omits $g$ from the functional and the constraint \eqref{constraintgeneral} is replaced by $\int_{B_R} \Delta f \, dV = 0$, which implies $f \in H^2_0(B_R)$ and so ${\mathcal E}(R) = E(B_R, 0) < 0$. 

Let us justify that this auxiliary problem, namely the minimization of  \eqref{twoballenergygeneral} under assumptions \eqref{constraintgeneral.11}--\eqref{constraintgeneral}, does indeed have a unique minimizer $(f,g)$. In Euclidean space, existence of a minimizer follows by the usual direct method, in view of the coercivity inequality
\begin{equation} \label{eq:Grisvard}
\int_{B_a}|\Delta f|^2 \, dV \geq \int_{B_a}|D^2 f|^2 \, dV \geq C_a \lVert f \rVert^2_{H^2(B_a)} , \qquad f \in H^2 \cap H^1_0(B_a) ,
\end{equation}
(and the analogous inequality for $g$), which we now explain. The first inequality in \eqref{eq:Grisvard} is a consequence of the following simple computation (see Grisvard \cite[Chapter 3]{Gr85}) in which partial derivatives are denoted by subscripts, the Einstein convention is used, and integrals over the boundary are taken with respect to surface measure: 
\begin{align*}
\int_{B_a} |\Delta f|^2 = \int_{B_a} f_{ii}f_{jj} 
& = \int_{\partial B_a} f_{ii}f_j n_j-\int_{B_a} f_{iij}f_j \\
& =\int_{\partial B_a} f_{ii}f_j n_j-\int_{\partial B_a} f_{ij}f_jn_i +  \int_{B_a} f_{ij}f_{ij} \\ 
& = \int_{\partial B_a} \! \Big(\frac{\partial f}{\partial n} \Big)^{\! \! 2} \mathcal H + \int_{B_a} |D^2 f|^2 ,
\end{align*}
where $\mathcal H \geq 0$ is the mean curvature of the boundary of $B_a$. (The chain of equalities should be proved first for smooth $f$ and then extended to general $f \in H^2 \cap H^1 _0(B_a)$. See also \cite{adolfsson} for more general results in this respect.) The second inequality in \eqref{eq:Grisvard} follows from standard Poincar\'{e}  inequalities, using that $\int_{B_a} f_i = 0$ when $f \in H^1_0(B_a)$. 

Existence of a minimizing pair $(f,g)$ in the spherical and hyperbolic cases can be justified similarly, after transforming the integrals of $|\Delta f|^2$ and $|D^2 f|^2$ on the geodesic ball $B_a$ to integrals on a Euclidean ball, via a global coordinate chart on $B_a$. Additional lower order terms arise in the integrals due to the expression for the Laplacian in the coordinate system (see \cite[Chapter 2]{Hebey1996}) and also due to the change of variable (Jacobian) factors in the integrals; these lower order terms can be absorbed in the coercivity estimate, so that again $\int_{B_a}|\Delta f|^2 \, dV \geq C_a \lVert f \rVert^2_{H^2(B_a)}$ as needed for existence of a minimizing pair.

The minimizing pair is unique, in view of the strict convexity of the integrands in the energy functional. Uniqueness implies that $f$ and $g$ are radial on their respective balls. 

Observe that the two-ball energy can be evaluated explicitly at the endpoints $a=0$ and $a=R$:  
\[
{\mathcal E}(0)=0 > E(B_R, 0) = {\mathcal E}(R) .
\] 
We will show in Step 4 that ${\mathcal E}(a)$ is minimal at the right endpoint $a=R$. Hence
\[
E(\Om, 0) \ge {\mathcal E}(a) \ge {\mathcal E}(R)= E(B_R, 0) ,
\]
which proves the minimality of the ball in Theorem \ref{as10general}. 

First we show the minimizing $f$ and $g$ are smooth and radially decreasing. 
\begin{proposition} \label{pr:eulerlagrange}
The minimizers $f$ and $g$ for the two-ball problem \eqref{twoballenergygeneral} under constraints \eqref{constraintgeneral.11}---\eqref{constraintgeneral} satisfy the Euler--Lagrange equations
\begin{equation} \label{eulerlagrange}
\Delta^{\! 2} f = 1 \ \text{in $B_a$} \qquad \text{and} \qquad \Delta^{\! 2} g = 0 \ \text{in $B_b$}
\end{equation}
and hence are smooth and radially strictly decreasing. 
\end{proposition}
\begin{proof}
The Euler--Lagrange equations follow immediately from the energy formula \eqref{twoballenergygeneral}. Elliptic regularity then ensures smoothness of $f$ and $g$ on their respective balls. 

We will show that they are radially decreasing. The equation $\Delta (\Delta f)=1>0$ means the radial function $\Delta f$ is strictly subharmonic. Expressing the Laplacian in the radial coordinate, we find  
\[
\frac{1}{(\sn r)^{N-1}} \frac{d\ }{dr} \! \left( (\sn r)^{N-1} \frac{d\ }{dr} \Delta f \right) > 0 
\]
 where 
 \[
\sn r = 
\begin{cases}
\sinh r & \text{when $M = \Hyp^N$,} \\
r & \text{when $M = \R^N$,} \\
\sin r & \text{when $M = \Sph^N$.}
\end{cases}
\]
 After multiplying the equation by $(\sn r)^{2(N-1)}$, it says
 \[
 \frac{d^2\ }{dt^2} \,  \Delta f > 0 
 \]
where the new variable $t$ is determined by
\[
\frac{d\ }{dt} = (\sn r)^{N-1} \frac{d\ }{dr} .
\]
This $t$ is a strictly increasing function of $r$, and $t \to -\infty$ as $r \to 0$ since $dt/dr \simeq r^{1-N} \geq r^{-1}$ for small $r$. 

We have shown $\Delta f$ is a strictly convex function of $t$. It tends to a finite limit as $t \to -\infty$, namely the value of $\Delta f$ at the origin, and so $\Delta f$ is necessarily strictly increasing as a function of $t$ and hence of $r \in (0,a)$. Further, $\int_{B_a} \Delta f \, dV \leq 0$ by the constraint \eqref{constraintgeneral}. Combining these two properties, we deduce that $\int_{B_r} \Delta f \, dV < 0$ for all $r \in (0,a)$; here one uses that either $\Delta f$ is negative for all $r \in (0,a)$ or else $\Delta f$ changes from negative to positive at some radius. 

Then the divergence theorem implies $\partial f/\partial r<0$, so that $f$ is strictly radially decreasing. 

The argument is easier for $g$: the equation $\Delta (\Delta g)=0$ means the smooth radial function $\Delta g$ is harmonic and hence is constant. The constant is $\leq 0$ due to the constraint \eqref{constraintgeneral} and so $\int_{B_r} \Delta g \, dV \leq 0$ for all $r \in (0,a)$, so that $\partial g/\partial r \leq 0$ and $g$ is radially decreasing. 
\end{proof}

\begin{remark}
An alternative proof that $f$ and $g$ are radially decreasing (at least in the Euclidean setting) can be based on an argument of Nadirashvili (see \cite[equation (3.12)]{Nad95}). The idea is to consider a competitor $\tilde{f}$ for $f$ built in such a way that $\partial_r\tilde{f}=-|\partial_r f|$. This procedure corresponds to flipping the graph of $f$ at the points where its derivative vanishes. In this fashion, we immediately get that $\tilde{f}$ is decreasing. Besides, not only $|\Delta f|=|\Delta \tilde{f}|$, but also $\tilde{f}\geq|f|$, so that replacing $f$ with $\tilde{f}$ does not increase the energy \eqref{twoballenergygeneral}. Lastly, $\partial_n\tilde{f}=-|\partial_n f|=\partial_n f$, the last equality coming from the sign assumption in \eqref{constraintgeneral}. As a result $(\tilde{f},g)$ is admissible, and so it is a minimizing pair for the problem \eqref{twoballenergygeneral}--\eqref{constraintgeneral}. By uniqueness, $f=\tilde{f}$ must be decreasing. Of course, the same approach can be applied for $g$. The fact that $f$ and $g$ are not locally constant follows from their analyticity, and hence $f$ and $g$ are strictly decreasing.
\end{remark}

\subsubsection*{\sc Step 4 --- minimizing the two-ball energy by transplantation} In order to show that the two-ball energy attains its minimum at $a=R$, we develop an argument based on  mass transplantation in which we enlarge the $a$-ball, shrink the $b$-ball, and show that the energy decreases. 

Namely, we fix $0 < a < R$ and prove that either ${\mathcal E}(a)>{\mathcal E}(R)$ or else  
\[
{\mathcal E}(a) > {\mathcal E}(a+\e) 
\]
for all small $\e >0$. Since ${\mathcal E}$ is  obviously continuous on $[0,R]$  (because the solutions of the Euler--Lagrange equation \eqref{eulerlagrange}, when rescaled over unit balls, solve an equation whose data is continuous with respect to $a$), this argument implies that ${\mathcal E}(a)$ is strictly minimal at $a=R$. 

Denote the pair of minimizers for the two-ball energy ${\mathcal E}(a)$ by $(f,g)$. Recall $f$ and $g$ are smooth, and are  radially decreasing and nonnegative and equal zero at the boundary. 

Case (i). First suppose $\int_{B_b} \Delta g \, dV = 0$, so that $\int_{B_a} \Delta f \, dV = 0$ by the constraint and hence $f \in H^2_0(B_a)$. Minimality of the pair $(f,g)$ in the definition of ${\mathcal E}(a)$ implies that $g \equiv 0$ on $B_b$, and so 
\begin{align*}
{\mathcal E}(a) 
& = \frac{1}{2} \int_{B_a} |\Delta f|^2 \, dV - \int_{B_a} f \, dV \\
& \geq  E(B_a,0) \\
& > E(B_R,0) = {\mathcal E}(R) 
\end{align*}
by strict domain monotonicity. Thus ${\mathcal E}(a)>{\mathcal E}(R)$, as we wanted to show. 

Case (ii). Suppose $\int_{B_b} \Delta g \, dV < 0$. Choose $\delta>0$ small enough that the integral remains negative whenever a set $S$ of measure less than $\delta$ is removed from the ball:  
\[
|S| < \delta \quad \Longrightarrow \quad \int_{B_b \setminus S} \Delta g \, dV < 0.
\] 
Given $0<\e<R-a$, choose $\e^\prime>0$ such that $V(B_{a+\e})+V(B_{b-\e^\prime})=V(B_R)$. Require $\e$ to be small enough that the annulus $B_b \setminus \ov B_{b-\e^\prime}$ has measure less than $\delta$. 

For the transplantation step, we will construct an admissible pair of functions $f_\e$ and $g_\e$ on $B_{a+\e}$ and $B_{b-\e^\prime}$, respectively. Start by letting 
\[
g_\e = g-g(b-\e^\prime) \qquad \text{on $B_{b-\e^\prime}$.} 
\]
Notice $g_\e$ equals zero at $r=b-\e^\prime$ and belongs to $H^1_0 \cap H^2(B_{b-\e^\prime})$, with $\Delta g_\e = \Delta g$. 

Next we construct $f_\e$ on the ball $B_{a+\e}$. Rearrange the values of $\Delta g$ on the annulus $B_b \setminus \ov B_{b-\e^\prime}$ to obtain a radially symmetric function $G^{**}$ on the annulus $B_{a+\e}\setminus \ov B_a$, with the rearrangement carried out in any manner that preserves the volume (measure) of superlevel sets. Write $h_\e$ for the weak solution of the Poisson equation 
\[
\Delta h_\e = -G^{**} \qquad \text{on the annulus $B_{a+\e}\setminus \ov B_a$}
\]
with Neumann condition $\partial h_\e/\partial r = \partial f/\partial r $ at the inner boundary $r=a$ and Dirichlet condition $h_\e=0$ at the outer boundary $r=a+\e$. Note $h_\e$ is radially symmetric since $G^{**}$ and the boundary conditions are radially symmetric. Define a radial function 
\[
f_\e
= \begin{cases}
f+h_\e(a) & \text{on $B_a$,} \\
h_\e & \text{on $B_{a+\e} \setminus B_a$,}
\end{cases}
\]
and observe $f_\e \in H^1_0 \cap H^2(B_{a+\e})$, noting that the values and slopes of $f_\e$ match from the left and right at radius $a$. The Laplacian is 
\[
\Delta f_\e 
= \begin{cases}
\Delta f & \text{on $B_a$,} \\
- G^{**} & \text{on $B_{a+\e} \setminus \ov{B_a}$} . 
\end{cases} 
\]
The admissibility constraint holds since 
\begin{align*}
\int_{B_{a+\e}} \Delta f_\e \, dV 
& = \int_{B_a} \Delta f \, dV - \int_{B_{a+\e}\setminus \ov B_a} G^{**} \, dV\\
& = \int_{B_b} \Delta g \, dV - \int_{B_b \setminus \ov B_{b-\e^\prime}} \Delta g \, dV \\
& = \int_{B_{b-\e^\prime}} \Delta g \, dV
\end{align*}
by the admissibility of $(f,g)$ and equimeasurability of $\Delta g$ and its rearrangement $G^{**}$. 

Now that the pair $(f_\e, g_\e)$ is known to be admissible, we obtain from the definition of the energy as a minimum that  
\begin{equation}
{\mathcal E}(a+\e) \leq \frac{1}{2} \int_{B_{a+\e}} |\Delta f_\e|^2 \, dV - \int_{B_{a+\e}} f_\e \, dV + \frac{1}{2} \int_{B_{b-\e^\prime}} |\Delta g_\e|^2 \, dV . \label{aplusepsilongeneral}
\end{equation}
On the right side, the transplantation procedure ensures that 
\[
\int_{B_{a+\e}} |\Delta f_\e|^2 \, dV+ \int_{B_{b-\e^\prime}} |\Delta g_\e|^2 \, dV = \int_{B_a} |\Delta f|^2 \, dV  +  \int_{B_b} |\Delta g|^2 \, dV .
\]
To handle the remaining term on the right, we will show  
\begin{equation} \label{eq:fstrict}
\int_{B_{a+\e}} f_\e \, dV > \int_{B_a} f \, dV ,
\end{equation}
so that \eqref{aplusepsilongeneral} implies ${\mathcal E}(a+\e) < {\mathcal E}(a)$, concluding the proof of Step 4. 

To show \eqref{eq:fstrict}, it suffices to show $h_\e$ is radially decreasing and positive on the annulus $B_{a+\e}\setminus \ov B_a$, because then $f_\e$ is an upward shift of $f \geq 0$. Let $c \in (a,a+\e)$. Write $P$ for perimeter. By integrating $\Delta h_\e$ over the annulus $\{ a < r < c \}$, we find 
\begin{align*}
& \frac{\partial h_\e}{\partial r} (c) \, P(\partial B_c) \\
& = \frac{\partial f}{\partial r} (a) \, P(\partial B_a) - \int_{B_{c} \setminus \ov B_a} G^{**} \, dV \quad \text{since $\Delta h_\e=-G^{**}$} \\
& = \int_{B_a} \Delta f \, dV - \int_S \Delta g \, dV \quad \text{where $S \subset B_b \setminus \ov B_{b-\e^\prime}$ is some suitable set} \\
& = \int_{B_b} \Delta g \, dV - \int_S \Delta g \, dV \qquad \text{by the constraint on $f$ and $g$} \\
& < 0 ,
\end{align*}
using that $S \subset B_b \setminus \ov B_{b-\e^\prime}$ has measure less than $\delta$. Thus $\partial h_\e / \partial r < 0$ at $c$ and so $h_\e$ is radially decreasing, which implies positivity since it was constructed to equal $0$ at $r=a+\e$. 

\subsubsection*{\sc Step 5 --- uniqueness of the minimizing domain} 
Suppose equality holds in Theorem \ref{as10general}. In view of the strict minimality of ${\mathcal E}(\cdot)$ at the right endpoint in Step 4 above, we must have $a=R$. Hence $\Omega_-$ is empty and $V(\Omega_+)=V(\Omega)$. In addition, equality must hold in Step 2 when Talenti's theorem is applied, and so $\Omega_+$ equals a geodesic ball $B$ up to a set of measure zero. (We may center the ball so that $B=B_R$.) Since $\Omega$ is open and $\Omega \setminus \overline{B}$ has measure zero, it follows that $\Omega \sq B$ and $B \setminus \Omega$ has measure zero. 


We go further and prove $B \setminus \Omega$ has $H^2$-capacity zero. After extending $u_\Om$ by $0$ to the rest of $B$, we see it is not only in $H_0^2(B)$ but is a minimizer of $E(B,0)$, due to the equality $E(\Om,0)=E(B,0)$. By uniqueness of this minimizer, we have that $u_\Om=u_B$ as $H^2$ functions, that is, up to a set of $H^2$-capacity zero. Since $u_B>0$ on $B$ by the proof of Proposition \ref{pr:eulerlagrange},
we conclude that up to an $H^2$-null set one has  
\[
B = \{u_B>0\}=\{u_\Om>0\}\subseteq\Omega, 
\]
the last inclusion coming from the fact that $u\in H_0^2(\Omega)$. Thus, up to the removal of a set of $H^2$-capacity zero, the ball is the only optimizer for Theorem \ref{as10general}.

\section{\bf Alternative (partial) proof of Theorem \ref{as10general} by explicit formulas}
\label{sec:explicitproof}

In this section we sketch a way to prove Theorem \ref{as10general} in the Euclidean case, and the $2$-dimensional spherical and hyperbolic cases, by explicitly evaluating the two-ball energy and then minimizing it. We have not found a similarly explicit proof in the higher dimensional spherical or hyperbolic situations. In any event, we think the transplantation method in Section \ref{sec:as10generalproof} is more illuminating than the following explicit approach, even with the appealingly exact formulas below.  

We state only the key formulas and leave proofs to the interested reader. 

\subsection{Euclidean case ($M=\RN, \sigma=0$)}
\label{as11euclidean}

In the Euclidean case, a direct computation using the Euler--Lagrange conditions \eqref{eulerlagrange} leads to the formulas
\[
f(r) = c \frac{a^2 - r^2}{2N} - \frac{a^4 - r^4}{8N(N+2)} , \qquad g(r) = d \frac{b^2 - r^2}{2N} .
\]
The constraint equation \eqref{constraintgeneral} enables us to find the constants $d$ in terms of $c$ as
\[
d = \frac{2(N+2)a^N c -a^{N+2}}{2(N+2)b^N} . 
\]
The choice of $c$ that minimizes the energy of $f$ and $g$ in \eqref{twoballenergygeneral} is  
\[
c = \frac{a^2 (N R^N+2(R^N-a^N))}{2N(N+2)R^N} , 
\]
where we used that $a^N+b^N=R^N$ by the volume normalization. (Recall $R$ is the radius of the ball $\Omega^*$.) Hence one may compute the two-ball energy as  
\[
{\mathcal E}(a) = - \frac{|\Sph^{N-1}|}{2} \frac{2 (N+2) R^N a^{N+4} -  (N+4) a^{2(N+2)}}{N^3 (N+2)^2 (N+4)R^N} .
\]
This energy has negative derivative with respect to $a$, since 
\begin{equation} \label{eq:energyderiv}
{\mathcal E}'(a)= - |\Sph^{N-1}| \, \frac{R^N-a^N}{N^3 (N+2)R^N} a^{N+3} <0 , \qquad a \in (0,R) , 
\end{equation}
and this direct estimate provides an alternative proof of Theorem \ref{as10general} in the Euclidean case. 

The minimum value of the energy is 
\[
E(\Omega^*,0) = {\mathcal E}(R) = - \frac{|\Sph^{N-1}|}{2} \, \frac{R^{N+4}}{N^2(N+2)^2 (N+4)} .
\]

\subsection{$2$-dim spherical case ($M=\Sph^2, \sigma=0$)}
\label{as11spherical}

On the sphere, the formulas that result from computing the two-ball energy minimizers explicitly are considerable more complicated than in the Euclidean case. We sketch below the results that can be obtained, mostly to convince the reader that such calculations are not the most insightful way to prove Theorem \ref{as10general} outside of Euclidean spaces. 

Consider the spherical case in $2$-dimensions ($N=2$). By solving the Euler--Lagrange equations \eqref{eulerlagrange} on spherical caps (geodesic disks), one can obtain that 
\[
\begin{split}
f(r) 
& = 2 (\log \cos(r/2))^2 - 2 (\log \cos(a/2))^2 \\
& \quad + \Li_2(-\tan^2(r/2)) - \Li_2(-\tan^2(a/2)) \\
& \quad + c \log \frac{\cos(r/2)}{\cos(a/2)} , \\
g(r) & = d \log \frac{\cos(r/2)}{\cos(b/2)} ,
\end{split}
\]
where $\Li_2(z)=\sum_{k=1}^\infty z^k/k^2$ is the polylogarithm function. The constraint equation \eqref{constraintgeneral} yields the constant $d$ in terms of $c$ as
\[
d = \frac{\sin^2(a/2) c - 4 \cos^2(a/2) \log \cos(a/2)}{\sin^2(b/2)} .
\]
The choice of $c$ that minimizes the energy of $f$ and $g$ in \eqref{twoballenergygeneral} is  
\[
c =  \frac{\cos R - \cos a + 2(1+\cos R) \log \cos(a/2)}{\sin^2(R/2)} ,
\]
where we used that $\sin^2(a/2)+\sin^2(b/2)=\sin^2(R/2)$ by the volume normalization for the $2$-sphere. The two-ball energy can then be evaluated as 
\[
\begin{split}
{\mathcal E}(a) 
& = - \frac{\pi}{4 \sin^2(R/2)}  \big( 1 - \cos(2a) - 4 (1 - \cos a) \cos R \\
& + 16 (\cos a - \cos R) \log \cos (a/2) - 16 (1 + \cos R) (\log \cos(a/2))^2 \big) . 
\end{split}
\]
This energy has derivative 
\[
{\mathcal E}'(a)= \pi \frac{\cos a - \cos R}{\sin^2(R/2)} \big( 1 - \cos a + 4 \log \cos(a/2) \big) \tan (a/2) .
\]
Note that ${\mathcal E}^\prime(a)<0$ when $a \in (0,\pi)$, because the middle factor $1 - \cos a + 4 \log \cos(a/2)$ is negative (it has value $0$ at $a=0$ and its derivative is $\sin a - 2 \tan(a/2) < 0$). Thus Theorem \ref{as10general} has been proved explicitly for the $2$-dimensional spherical case. 

The minimum value of the energy in this spherical case is 
\[
E(\Omega^*,0) = {\mathcal E}(R) = 2\pi \big( 4 \cot^2(R/2) (\log \cos(R/2))^2 -\sin^2(R/2) \big)
\]
where $R$ is the radius of the geodesic disk $\Omega^*$ on the $2$-sphere.

\subsection{$2$-dim hyperbolic case ($M=\Hyp^2, \sigma=0$)}
\label{as11hyperbolic}
For the $2$-dimensional hyperbolic space $\Hyp^2$, the formulas are analogous to the spherical case and so we simply state them below: 
\[
\begin{split}
f(r) 
& = 2 (\log \cosh(r/2))^2 - 2 (\log \cosh(a/2))^2 \\
& \quad + \Li_2(\tanh^2(r/2)) - \Li_2(\tanh^2(a/2)) \\
& \quad + c \log \frac{\cosh(r/2)}{\cosh(a/2)} , \\
g(r) & = d \log \frac{\cosh(r/2)}{\cosh(b/2)} ,
\end{split}
\]
\[
\begin{split}
d & = \frac{\sinh^2(a/2) c + 4 \cosh^2(a/2) \log \cosh(a/2)}{\sinh^2(b/2)} , \\
c & = - \frac{\cosh R - \cosh a + 2(1+\cosh R) \log \cosh(a/2)}{\sinh^2(R/2)} ,
\end{split}
\]
\[
\begin{split}
{\mathcal E}(a) 
& = \frac{\pi}{4 \sinh^2(R/2)}  \big( 1 - \cosh(2a) - 4 (1 - \cosh a) \cosh R \\
& + 16 (\cosh a - \cosh R) \log \cosh (a/2) - 16 (1 + \cosh R) (\log \cosh(a/2))^2 \big) , \\
{\mathcal E}'(a)& = \pi \frac{\cosh R - \cosh a}{\sinh^2(R/2)} \big( 1 - \cosh a + 4 \log \cosh(a/2) \big) \tanh (a/2) .
\end{split}
\]
Again ${\mathcal E}^\prime(a)<0$ when $a>0$, because the middle factor $1 - \cosh a + 4 \log \cosh(a/2)$ is negative (it has value $0$ at $a=0$ and its derivative is $-\sinh a + 2 \tanh(a/2) < 0$). Thus Theorem \ref{as10general} follows for the $2$-dimensional hyperbolic case. 

The minimum value of the energy in this hyperbolic case is 
\[
E(\Omega^*,0) = {\mathcal E}(R) = 2\pi \big( 4 \coth^2(R/2) (\log \cosh(R/2))^2 -\sinh^2(R/2) \big) .
\]

\section{\bf Proof of Theorem \ref{as13} --- positive compression in Euclidean plane}
\label{sec:compression}

We work in $N$ dimensions, in order to identify the step where the proof breaks down for $N \geq 3$. Assume $V(\Omega)=|B_1|$, so that $R=1$ in what follows. 

\subsection*{Two-ball energy ${\mathcal E}(a, \sigma)$}
By analogy with the zero compression case in Section \ref{sec:as10generalproof}, and again relying on the Talenti-style result Theorem \ref{talenti}, let us denote by ${\mathcal E}(a, \sigma)$ the energy for the following two-ball problem with ball radius $a \in (0,1)$ and compression parameter $\sigma>0$:
\begin{align*}
{\mathcal E}(a, \sigma) = \min \Big ( \frac{1}{2} \int_{B_a} |\Delta f|^2  \, dV 
& - \frac{\sigma ^2}{2} \int_{B_a} |\nabla f|^2 \, dV - \int_{B_a} f \, dV \\
& + \frac{1}{2} \int_{B_b} |\Delta g|^2 \, dV - \frac{\sigma ^2}{2} \int_{B_b} |\nabla g|^2 \, dV\Big)
\end{align*}
where the minimum is taken over ``admissible'' pairs $(f,g)$, meaning: 
\begin{align*}
f \in H^1_0 \cap H^2(B_a), & \qquad g \in H^1_0 \cap H^2(B_b) , \\
a^{N-1} \left. \frac{\partial f }{\partial n} \right|_{\partial B_a} &  =  b^{N-1} \left. \frac{\partial g}{\partial n} \right|_{\partial B_b}= \text{constant}. 
\end{align*}

We start with the following observations for the resolution of a generic problem. Let $r >0$ and suppose $0 < \sigma^2 < \Lambda_1(B_r)=$\, buckling load of the ball (this buckling load scales as $\Lambda_1(B_r)=r^{-2}\Lambda_1(B_1)$.) Given $c\in \R$ and $\alpha \in \{0,1\}$ we consider
\[
\min  \Big (\frac{1}{2} \int_{B_r} |\Delta \psi |^2  \, dV - \frac{\sigma ^2}{2} \int_{B_r} |\nabla \psi|^2 \, dV - \alpha \int_{B_r} \psi \, dV \Big)
\]
under the conditions $\psi  \in H^1_0 \cap H^2(B_r), \frac{\partial \psi }{\partial n}=c$ on $\partial B_r$. We claim the following. 
\begin{itemize}
\item A minimizer $\psi$ exists provided $\sigma$ is sufficiently small.  Indeed, existence follows by the usual direct method, noting that coercivity follows from the Poincar\'{e} inequality in $H^2_0 (B_r)$ applied to $\psi- \frac{c}{2} (|x|^2-r^2)$. 
 \item The minimizer $\psi$ is smooth and satisfies the Euler--Lagrange equation 
\begin{equation}
\label{as01.2}
\begin{cases}
\Delta^{\! 2} \psi+\sigma^2 \Delta \psi=\alpha&\hbox{in }B_r, \\
\psi= 0&\hbox{on }\partial B_r ,  \\
 \frac{\partial \psi}{\partial n} =c&\hbox{on }\partial B_r.
\end{cases}
\end{equation}
\item The minimizer $\psi$ is unique. For if $\psi_1, \psi_2$ are two minimizers, then $\psi_1-\psi_2 \in H^2_0(B_r)$ can be taken as a test function for both their Euler--Lagrange equations. Taking the difference, one finds 
\[
\int_{B_r} |\Delta (\psi_1-\psi_2)|^2 \, dV - \sigma^2\int_{B_r} |\nabla (\psi_1-\psi_2)|^2 \, dV =0.
\]
 From the restriction $0 < \sigma ^2 <\Lambda_1(B_r)$, we conclude $\psi_1-\psi_2 \equiv 0$.
 \item The minimizer $\psi$ is radially symmetric in $B_r$, from uniqueness. 
 \end{itemize}

A consequence of these observations is that provided $\sigma$ is sufficiently small, one has for each $a \in (0,1)$ that   minimizers $(f,g)$ for the two-ball energy exist and are radially symmetric, and satisfy the strong form of their respective Euler--Lagrange equations. We do not know at this point whether the pair of minimizers is unique: different values of the constant normal derivative could conceivably result in different pairs giving the same value for the energy. 

\subsection*{Unique minimizer for the two-ball energy} Note first that if $0\le \sigma_1 \leq \sigma_2$ then ${\mathcal E}(a, \sigma_1)\ge {\mathcal E}(a, \sigma_2)$. 

We shall prove the minimizing pair $(f,g)$ is unique by finding the exact values of $\e=-\frac{\partial f }{\partial n} > 0$  and $\delta =-\frac{\partial g }{\partial n} > 0$, with $a^{N-1} \e= b^{N-1}\delta$. 

Write $v_a$ for the function $\psi$ in \eqref{as01.2} with $r=a, \alpha=1$, and $c=0$, so that $h_a= (v_a-f)/\e$ corresponds to the function $\psi$ with  $r=a, \alpha=0,$  and $c=1$. Then $v_a$ and $h_a$ are radially symmetric. Similarly, with the same notation for $h_b$ we have $h_b = -g/\delta$. 

Substituting $f=v_a-\e h_a$ and $g=-\delta h_b$ into the energy and then integrating by parts (pushing the Laplacian in the cross terms onto $h_a$), we get
\[
{\mathcal E}(a, \sigma) = \min _{a^{N-1} \e= b^{N-1}\delta} \left( \frac {\e^2}{2}\int_{\partial B_a} \Delta h_a + \e \int_{B_a} h_a - \frac 12 \int_{B_a}v_a + \frac {\delta^2}{2}\int_{\partial B_b} \Delta h_b \right) .
\]
Minimizing with respect to $\e$ (with $\delta$ determined from the constraint $a^{N-1} \e= b^{N-1}\delta$), we find 
\[
\e = - \frac{\int_{B_a} h_a}{\int_{\partial B_a} \Delta h_a + (a/b)^{2N-2} \int_{\partial B_b} \Delta h_b}
\]
and hence the minimal energy is 
\[
2{\mathcal E}(a, \sigma) = - \frac{\left( \int_{B_a} h_a \right)^{\! 2}}{\int_{\partial B_a} \Delta h_a + (a/b)^{2N-2} \int_{\partial B_b} \Delta h_b} - \int_{B_a}v_a .
\]
The denominator is positive since $ \Delta h_a|_{\partial B_a}$ and $ \Delta h_b|_{\partial B_b}$ are positive constants. Indeed, we have  that $\Delta h_a+\sigma^2h_a=\tilde{c}$ for some $\tilde c$. On the boundary of the ball, $\Delta h_a=\tilde c$, and so it is enough to prove that $\tilde{c} >0$. Integrating $\Delta h_a+\sigma^2h_a=\tilde{c}$ on the ball and using $\partial_nh_a=1$ on $\partial B_a$, we get that $|\partial B_a|+ \sigma^2\int _{B_a} h_a= \tilde{c}|B_a|$.   If $\int_{B_a} h_a \geq 0$ then $\tilde{c}>0$. Suppose instead that $\int_{B_a} h_a <0$. Multiply $\Delta h_a+\sigma^2h_a=\tilde c$ by $h_a$ and integrate over $B_a$, obtaining that $-\int |\nabla h_a|^2+ \sigma^2 \int h_a^2= \tilde c\int h_a$. The left side is negative for sufficiently small $\sigma$, as a consequence of the Poincar\'e inequality in $H^1_0(B_1)$, and so again $\tilde{c}>0$. 

In fact we can prove $h_a<0$ on $B_a$, and hence $ \int_{B_a} h_a<0$. Note the boundary conditions imply that the radial function $h_a$ is negative near the boundary $\partial B_a$, and so if it changed sign, there would be a positive radius $r$ at which $h_a<0$ and $\nabla h_a=0$. The function $H(x)=h_a(x)-h_a(r)$ would then belong to $H^2_0(B_r)$ and satisfy $\Delta ^2 H+ \sigma ^2 \Delta H=0$, contradicting our choice of $\sigma ^2 < \Lambda_1 (B_r)$.

We conclude that the minimizing pair $(f,g)$ is unique. 

\subsection*{Derivative of the energy at $a=1$} 
The energy curve is flat at the right endpoint where the $a$-ball contains all the volume. 
\begin{proposition}\label{as13.1}
There exists $\sigma_N >0$ such that for all $0<\sigma \le \sigma_N$ we have
\[
\left. \frac{\partial\ }{\partial a}   {\mathcal E}(a, \sigma) \right|_{a=1} =0.
\]
\end{proposition}
\begin{proof}
Assume $0<\sigma ^2_N < 2^{2/N} \lb_1(B_1)$ where $\lb_1(B_1)$ is the first eigenvalue of the Dirichlet Laplacian of the unit ball. Since $2^{2/N} \lb_1(B_1) \leq \lambda_2(\Omega) \leq \Lambda_1(\Omega)$ by Krahn's inequality \cite[Theorem 4.1.1]{Henrot2006} and Payne's inequality \cite[Theorem 4.6.2]{Kesavan2006}, the assumption ensures that $\sigma^2 <\Lambda_1(\Om)$. 

By standard shape derivatives we have
\[
\left. \frac{d\ }{da} \right|_{a=1}  \int_{B_a} v_a = |\Delta v_1|^2_{\partial B_1} \, |\partial B_1|.
\]
To derive this equation, one could consult the general view of shape derivative formulas for the bi-Laplacian equation in \cite[Section 6]{BuKe22}. However, the formula is a quite direct consequence of the fact that 
\begin{align*}
 \int_{B_a} v_a 
 & =-2 \min_{u \in H^2_0(B_a)} \Big [\frac 12 \int_{B_a} |\Delta u|^2-\frac{\sigma ^2}{2}    \int_{B_a} |\nabla u|^2 - \int _{B_a} u \Big ] \\ 
& = -\min_{u \in H^2_0(B_1)}  \left( a^{N-4} \int_{B_1} |\Delta u|^2 - a^{N-2} \sigma^2 \int_{B_1} |\nabla u|^2 - 2a^N \int _{B_1} u \right) 
\end{align*}
by rescaling. Indeed, the last equation gives  
\begin{align*}
\left. \frac{d\ }{da}  \right|_{a=1}  \int_{B_a} v_a 
& = - \left( (N-4) \int_{B_1} |\Delta v_1|^2- (N-2)\sigma^2 \int_{B_1} |\nabla v_1|^2 - 2N \int _{B_1} v_1 \right) \\
& =  |\Delta v_1|^2_{\partial B_1} \, |\partial B_1| 
\end{align*}
where the first equality sign holds as the derivative of a minimum functional having a unique minimizer and the second equality sign is directly obtained by testing the equation satisfied by $v_1$ with the trial function $x \mapsto \sum_{i=1}^N x_i \frac{\partial v_1}{\partial x_i}$ (details omitted). 

Since $2 {\mathcal E}(1, \sigma)=- \int_{B_1} v_1$, it remains to compute the derivative of the other term in the energy, which in terms of difference quotients means we want to evaluate 
\begin{equation} \label{eq:big}
\lim_{a \nearrow 1} \, \frac{1}{a-1} \, \frac{  \left( \int_{B_a} h_a \right)^{\! 2} }{\int_{\partial B_a} \Delta h_a + (a/b)^{2N-2} \int_{\partial B_b} \Delta h_b}.
\end{equation}
The integrals over $B_a$ and $\partial B_a$ behave as expected in the limit as $a \nearrow 1$, with 
\[
\int_{B_a} h_a \to \int_{B_1} h_1 \qquad \text{and} \qquad \int_{\partial B_a} \Delta h_a \to  \int_{\partial B_1} \Delta h_1.
\]
To further evaluate that first limit, taking $v_1$ as a test function in the equation for $h_1$ and vice-versa, we get by subtraction that 
\[
\int_{B_1} h_1 = - \int_{\partial B_1} \Delta v_1 = -(\Delta v_1)|_{\partial B_1} \, |\partial B_1| .
\] 

To evaluate \eqref{eq:big}, we still need  
\[
\lim_{a \nearrow 1} \, (a-1) (a/b)^{2N-2} \int_{\partial B_b} \Delta h_b .
\]
We blow up around $b=0$ by introducing $\tilde{h}_b(x) = b^{-1} h_b(bx)$, which is the function $\psi$ from \eqref{as01.2} associated to $(b\sigma)^2$, $r=1$, $\alpha=0$, $c=1$. With this rescaling, 
\[
\int_{\partial B_b} \Delta h_b=b^{N-2} \int_{\partial B_1} \Delta \tilde{h}_b.
\]
Note that $\tilde{h}_0(x)= (|x|^2-1)/2$. Hence 
\begin{align*}
\lim_{a \nearrow 1} \, (a-1) \frac{a^{2N-2}}{b^{2N-2}} \int_{\partial B_b} \Delta h_b
& = -  \lim_{a \nearrow 1} \, a^{2N-2} \frac{1-a}{1-a^{N}}\int_{\partial B_1} \Delta \tilde{h}_b \\
& =- \frac 1N \int_{\partial B_1}  \Delta \tilde{h}_0 = - |\partial B_1| .
\end{align*}

Putting it all together, we find 
\[
 \left. \frac{\partial\ }{\partial a}   {\mathcal E}(a, \sigma) \right|_{a=1}  
 = - \frac{\left( -(\Delta v_1)|_{\partial B_1} \, |\partial B_1| \right)^2}{- |\partial B_1| } - (\Delta v_1)^2|_{\partial B_1} \, |\partial B_1| = 0 ,
\]
which concludes the proof. 
\end{proof}

Let us focus on the following expression for the energy
\begin{equation} \label{eq:energyexp}
2{\mathcal E}(a, \sigma) = -\frac{(1-a^{N})   \left( \int_{B_a} h_a \right)^{\! 2}  }{(1-a^{N})\int_{\partial B_a} \Delta h_a + a^{2N-2} \int_{\partial B_1} \Delta \tilde{h}_b} -\int_{B_a} v_a.
\end{equation}
For $a \in (0,1]$, all quantities in the expression depend smoothly on $a$, except for $\int_{\partial B_1} \Delta \tilde{h}_b$. The function $\tilde{h}_b$ solves the partial differential equation \eqref{as01.2} on $B_1$ with $\alpha=0,c=1$, and with coefficient $b^2 \sigma^2$ in place of $\sigma^2$. Consequently, $\tilde{h}_b$ may depend nonsmoothly on $a$ since $b^2= (1-a^N)^{2/N}$ is not smooth at $a=1$ when $N \geq 3$. 

Smoothness does hold when $N=2$. Indeed, by \eqref{eq:energyexp} the energy can be written in the form
\[
{\mathcal E}(a, \sigma)= \frac{(1-a) f_1(a, \sigma)}{f_2(a, \sigma)}+ f_3(a, \sigma)
\]
where each $f_i(a, \sigma)$ is smooth on $[\frac 12, 1] \times [-\epsilon, \epsilon]$, for some $\epsilon>0$, and $f_2(1, 0) \neq 0$. Hence:
\begin{lemma}\label{as13.2}
If $N=2$ then ${\mathcal E}(\cdot, \sigma) \to {\mathcal E}(\cdot, 0)$ in $C^2[\frac 12, 1]$ and $C^0[0, 1]$, as $\sigma \searrow 0$.
\end{lemma}
%
%

Note that $a \mapsto {\mathcal E}(a,0)$ is decreasing on $[0,1]$ by formula \eqref{eq:energyderiv} and is convex on the interval $[\Big (\frac{N+3}{2N+3}\Big)^{\!\! 1/N}, 1]$ by differentiating that formula. Moreover, the convexity is uniform on a slightly smaller interval, with 
\[ \frac{\partial ^2{\mathcal E}}{\partial a^2}(a,0) \ge C_N \quad \text{on} \quad \left[\Big (\frac{N+4}{2N+3}\Big)^{\! \! 1/N}, 1\right].\]

\subsection*{Proof of Theorem \ref{as13}}
Let $N=2$. At $\sigma =0$, we have by above that $\frac{\partial ^2{\mathcal E}}{\partial a^2}(a,0) \ge C_2>0$ when $a \in [\sqrt{6/7}, 1]$. From Lemma \ref{as13.2}, there exists $\sigma_2>0$ such that for all $\sigma \in (0, \sigma_2]$ and $a \in [\sqrt{6/7}, 1]$ we have 
\[
\frac{\partial ^2{\mathcal E}}{\partial a^2}(a,\sigma) \ge \frac {C_2}{2}.
\]
Hence ${\mathcal E}(a, \sigma)$ is a convex function of $a \in [\sqrt{6/7}, 1]$, when $\sigma \in [0, \sigma_2]$. Since, from Proposition  \ref{as13.1} the partial derivative at $a=1$ vanishes, we conclude that $a=1$ provides the minimum of ${\mathcal E}(a, \sigma)$ on $[\sqrt{6/7}, 1]$. 

In order to conclude the proof, it is enough to decrease $\sigma_2$ such that the minimum of ${\mathcal E}(a, \sigma)$ for $a \in [0, \sqrt{6/7}], \sigma \in [0,\sigma_2]$, is strictly greater than the maximum of ${\mathcal E}(1,\sigma)$ for $\sigma \in [0,\sigma_2]$, which can be done by continuity since ${\mathcal E}(a,0)$ achieves its strict minimum at $a=1$.
 
Lastly, suppose equality holds in Theorem \ref{as13}; we aim to show $\Omega$ is a disk. Strict minimality of ${\mathcal E}(\cdot,\sigma)$ at the right endpoint implies that $a=1$. Hence $\Omega_-$ is empty and $V(\Omega_+)=V(\Omega)$. Equality must hold in the argument above when Talenti's theorem is applied, and  so by arguing as in Step 5 of the proof of Theorem \ref{as10general}, one deduces that $\Omega \subseteq B$ and $B \setminus \Omega$ has $H^2$-capacity zero, using the fact established below that $u_B$ is positive on $B$. Note here that in $2$ dimensions, sets of $H^2$-capacity zero are necessarily empty, and so $\Om=B$. 

To finish the proof, it remains to show $u_B>0$. 
 
Since $E(\Omega,\sigma) = E(B,\sigma)$ by the equality assumption, uniqueness of energy minimizers implies that the minimizer $u_\Omega$ (extended by $0$ so that it belongs to $H^2_0(B)$) is also the minimizer $u_B$ for $B$. Let us show $u_B>0$. We may suppose $B$ is centered at the origin. Since $\Delta( \Delta u_B + \sigma^2 u_B)=1$ and $u_B$ is radial, there exists a constant $c$ such that  
\[
\Delta u_B(x)+ \sigma ^2 u_B(x) = c + \frac{|x|^2}{ 4}.
\]
Integrating over $B$ and keeping in mind that $\int_B \Delta u_B=0$ (since $\partial_n u=0$ on the boundary) while $\int_B u_B >0$ by \eqref{deflectionpositive}, we find $c+ 1/8>0$. Thus $\Delta u_B$ is positive on $\partial B$, and hence also on a neighborhood of $\partial B$. 
Assume for the sake of obtaining a contradiction that $u_B \leq 0$ somewhere in $B$. Writing $r= |x|$ and identifying $u_B(x)=u_B(r)$, there must exist radii $ 0\le r_1 <r_2 <1$ such that $u_B^\prime(r) > 0$ when $r \in (r_1, r_2)$ and $u_B^\prime(r_1)=u^\prime_B(r_2)=0$. By ``flipping'' the values when $r \in (r_1,r_2)$ and suitably increasing the values when $r\in (0,r_1)$, we construct a new function
\[
u(r)= u_B(r)|_{B\setminus B_{r_2}}+ (2u_B(r_2)-u_B(r))|_{B_{r_2}\setminus B_{r_1}}+ (2u_B(r_2)-2u_B(r_1)+u_B(r))|_{B_{r_1}} .
\]
Notice $u$ is continuous at $r_1$ and $r_2$, and $u^\prime$ equals $0$ there, so that $u \in H^2_0(B)$. It satisfies $|\nabla u_B|= |\nabla u|$ and $|\Delta u_B|= |\Delta u|$ a.e., and $\int _B u_B <\int_B  u$. But that is impossible since $u_B$ minimizes the energy $E(B,\sigma)$. Hence $u_B>0$ on $B$. 

\begin{remark}\label{as14}
Numerical computations indicate for $N=2$ that $\sigma_2 \approx 3.0$, while the first buckling eigenvalue of the unit disk occurs at   $\sigma \simeq 3.83$.
\end{remark}

\section{\bf The compliance problem with nonconstant load}
\label{sec:compliance}
 
 Let now discuss the general compliance maximization problem \eqref{as03} for nonconstant loads and no compression ($\sigma=0$):
 \begin{equation}\label{abl1000}
 \max \{\MC(\Om, \rho) : \Om \sq \R^N, \, V(\Om)=m, \, \rho \in L^\infty(\Om, [-1,1])\}.
 \end{equation}
 Alternatively, this problem can be written 
\begin{equation}\label{abl30}
  \min \{\frac 12 \int_\Om |\Delta u|^2-\int_\Om u \rho : \Om \sq \R^N, \, V(\Om)=m, \, \rho \in L^\infty(\Om, [-1,1]), u \in H^2_0(\Om)\}.
 \end{equation}
Consequently, if $\Om$ is given, the optimal load on $\Om$ satisfies $\rho_\Om^{opt} = \sign (u_{\Om }^{opt})$ so that $\int_\Om \rho_\Om^{opt} u_{\Om }^{opt}=\int _\Om |u_{\Om }^{opt}|$. 
In other words, the compliance equals the {\it total} deflection of the plate provided the load is optimal, maximizing the compliance. 
 
 We proceed to resolve the compliance problem in the planar Euclidean case. 

\begin{theorem} \label{th:total}
Let $N= 1, 2$, and suppose $\Om \subset \R^N$ is an open set with finite measure. For every measurable function $\rho :\Om\to[-1,1]$ we have
\begin{equation}\label{abll01}
\MC(\Om, \rho) \le\MC(\Om^*,1).
\end{equation}
 If equality occurs, then $\rho=1$ a.e.\ or $\rho=-1$ a.e., and $\Om=\Om^*$. 

\end{theorem}
An alternative formulation of the conclusion is
\[
\begin{split}
-2 \Big (\min_{u \in H^2_0(\Om) } \frac 12 \int_\Om |\Delta u|^2-\int_\Om u \rho\Big) 
& =\int _\Om \rho u_{\Om,\rho} \\
& \le \int _{\Om^*}  u_{\Om^*,1} 
=-2 \Big (\min_{u \in H^2_0(\Om^*) } \frac 12 \int_{\Om^*} |\Delta u|^2-\int_{\Om^*} u  \Big).
\end{split}
\]
\begin{proof}
We note first that functions in $H^2(\R^N)$ are continuous for $N=1,2$, and hence $u_{\Om,\rho}$ is continuous. The choice of $\rho \in L^\infty (\Om, [-1,1])$ may mean that the set $\{u_{\Om, \rho}=0\}$ has positive measure, but that is not an obstruction to the following proof. 

Based on Talenti's theorem \cite{Tal79} and on the observation that $\rho_\Om^{opt}= \sign (u_\Om^{opt})$, a two-ball analysis leads to the fact that 
\[
\min_{u \in H^2_0(\Om) } \frac 12 \int_\Om |\Delta u|^2-\int_\Om u \rho \ge \min_{a\in [0,R]} \mathcal E^{abs} (a),
\]
where 
\begin{equation}
\label{twoballenergygeneral.b}
{\mathcal E}^{abs}(a) = \min \left( \frac{1}{2} \int_{B_a} |\Delta f|^2 \, dV - \int_{B_a} f \, dV + \frac{1}{2} \int_{B_b} |\Delta g|^2 \, dV - \int_{B_b} g \, dV\right) ,
\end{equation}
with the minimum being taken over ``admissible'' pairs $(f,g)$, meaning: 
\begin{align}
f \in H^1_0 \cap H^2(B_a), & \qquad g \in H^1_0 \cap H^2(B_b) , \notag \\
a^{N-1} \left. \frac{\partial f }{\partial n} \right|_{\partial B_a} &  =  b^{N-1} \left. \frac{\partial g}{\partial n} \right|_{\partial B_b}= \text{constant}. \label{constraintgeneral.b}
\end{align}
 The key difference from the two-ball energy \eqref{twoballenergygeneral} for the constant load case is that here we do not know the sign of the load $\rho$ and so in \eqref{twoballenergygeneral.b} we must subtract the additional term $\int_{B_b} g \, dV$.

Note that $a^N+b^N \leq R^N$ where $R$ is the radius of the ball $\Omega^*$. We may suppose equality holds, meaning that $a^N+b^N = R^N$,  since the minimal two-ball energy that we find below has the form $-(\text{const.})R^{N+4}$ and thus decreases as $R$ gets bigger.

Arguing as in Section \ref{sec:compression}, one sees that a minimizing pair $(f,g)$ exists and that these functions are radially symmetric, satisfying $\Delta^2 f = 1$ in $B_a$ and $\Delta^2 g = 1$ in $B_b$, respectively. Direct computation leads to the formulas
\[
f(r) = c \frac{a^2 - r^2}{2N} - \frac{a^4 - r^4}{8N(N+2)} , \qquad g(r) = d \frac{b^2 - r^2}{2N} - \frac{b^4 - r^4}{8N(N+2)} ,
\]
for some constants $c$ and $d$. The constraint equation \eqref{constraintgeneral.b} evaluates to 
\begin{equation} \label{constraintgeneral.c}
2(N+2)a^N c -a^{N+2} = 2(N+2)b^N d - b^{N+2} , 
\end{equation}
which gives a linear relation between $c$ and $d$. The energy is then 
\[
\begin{split}
{\mathcal E}^{abs}(a) = \frac{\om_{N-1}}{8N^2} \min_{c,d}
\Big( 
& a^N \Big( \frac{N+6}{(N+2)(N+4)} a^4 - 4a^2 c + 4N c^2 \Big) 
\\
+ 
& 
b^N \Big( \frac{N+6}{(N+2)(N+4)} b^4 - 4b^2 d + 4N d^2) \Big) 
\Big)
\end{split}
\]
where $\om_{N-1}$ is the area of the unit sphere in $\R^N$. Thus the energy is a positive quadratic in $c$ and $d$. 

Minimizing with respect to $c$ and $d$ by Lagrange multipliers, subject to the constraint, gives 
\begin{equation} \label{eq:cd}
c = \frac{N a^{N+2} + (N+2) a^2 b^N + 2 b^{N+2}}{2 N(N+2)R^N} , \qquad d = \frac{2 a^{N+2} + (N+2) a^N b^2 + N b^{N+2}}{2 N(N+2)R^N} ,
\end{equation}
and hence 
\[
{\mathcal E}^{abs}(a) = 
- \om_{N-1} \frac{Na^{2N+4} + 
  2 a^N b^N ((N+2)(a^4 + b^4) + (N+4)a^2 b^2) + N b^{2N+4}}{2N^3 (N+2)^2 (N+4)R^N} .
\]

Suppose $N=1$. The energy formula gives 
\[
{\mathcal E}^{abs}(a) = - \frac{a^6 + 6 a^5 b + 10 a^3 b^3 + 6 a b^5 + b^6}{45R} = - \frac{R^6 - 15 R^2 (ab)^2 + 20 (ab)^3}{45R},
\]
where we simplified using $a+b=R$. The numerator of this simplified formula is positive and strictly decreasing with respect to $ab \in (0,R^2/4)$. Thus the energy is strictly minimal when $ab=0$, that is, when $a=0$ or $a=R$, both of which are one-ball situations. The minimal energy is $-R^5/45$.

Now suppose $N=2$. The energy formula simplifies to   
\[
{\mathcal E}^{abs}(a) = - \frac{\pi}{384} (a^2 + b^2)^3 = - \frac{\pi R^6}{384}
\]
since $a^2+b^2=R^2$. This energy is constant (independent of $a$) and so in particular attains its minimum at $a=R$ and at $a=0$. These are the desired one-ball configurations, completing the proof of the inequality in the theorem.

Suppose equality occurs in \eqref{abll01} for $N=1$ or $2$. Write $u=u_{\Om,\rho}$. Equality in the argument leading to the two-ball problem implies that $\rho=1$ a.e.\ on $\{u>0\}$ and $\rho=-1$ a.e.\ on $\{u<0\}$. For $N=1$, we saw above that equality forces one of those sets to have measure zero, meaning $a=0$ or $R$. We will show the same for $N=2$. So suppose $0<a<R$. Then when applying Talenti's theorem, equality holds for both the positive and negative parts, and so (by the equality case of that theorem \cite{Kesavan2006}) the sets $\{u>0\}$ and $\{u<0\}$ are each balls, possibly with sets of measure zero removed. Thus $\Om$ is the union of two disjoint balls, up to sets of measure zero. Since $u \in H^2_0(\Om)$ is the $H^2$-limit of a sequence of compactly supported smooth functions in those balls, we conclude that $\Delta u$ integrates to zero over each ball. Therefore Talenti's rearrangement of $-\Delta u$ forces the constant in \eqref{constraintgeneral.b} to equal zero. Hence the left side of \eqref{constraintgeneral.c} equals zero, giving $c=a^2/8$ when $N=2$. But when $N=2$, simplifying \eqref{eq:cd} with the help of $a^2+b^2=R^2$ yields that $c=R^2/8$, a contradiction. Hence $a$ must equal $0$ or $R$. 

Suppose $a=R$, so that $\{u<0\}$ has measure zero and so $\rho=1$ a.e.\ on $\Omega$, by above. Hence Theorem \ref{as10general} applies. The equality case of that theorem says that $\Om$ equals a ball, since  in dimensions $N=1,2$, any points of the ball omitted from $\Om$ would have positive $H^2$ capacity. Similarly, if $a=0$ then $\rho=-1$ a.e.\ and $\Omega$ equals a ball.  
\end{proof}

\subsection*{Remarks and open problems for variable loads}
In dimension $N \geq 3$, the energy of the balls  problem \eqref{twoballenergygeneral.b} seems from numerical work to be minimal for two equal balls, so that the conclusion of Theorem \ref{th:total} cannot be achieved by this method. For example, when $N=3$ the energy equals
\[
{\mathcal E}^{abs}(a) = - 2\pi \frac{3 a^{10} + 10 a^7 b^3 + 14 a^5 b^5 + 10 a^3 b^7 + 3 b^{10}}{4725R^3}
 \]
with $a^3+b^3=R^3$. Plotting as a function of $a$, one finds that the minimum occurs at $a=(1/2)^{1/3} R$, that is, when the two balls have equal size, in sharp contrast to the single ball we found in Theorem \ref{th:total} for $N=1,2$.

We can formulate the following open problems.

\subsection*{Open problem 1} 
Extend Theorem \ref{th:total} to dimensions $N \geq 3$, by proving that the maximum of the compliance in \eqref{abl1000} is attained on a ball with constant load $1$.

\subsection*{Open problem 2} Prove that the maximum of the {\it mean} deflection
\[
\max \left\{ \int_\Om u_{\Om, \rho} \, dV: \Om \sq \R^N, V(\Om)=m \text{ and } \rho \in L^\infty(\Om, [-1,1]) \right\}
\]
is attained on the ball with the constant load $\rho \equiv 1$ in dimensions $N\ge 2$. 

\subsection*{Open problem 3} Prove that the maximum of the {\it total } deflection
\[
\max \left\{ \int_\Om |u_{\Om, \rho}| \, dV: \Om \sq \R^N, V(\Om)=m \text{ and } \rho \in L^\infty(\Om, [-1,1]) \right\}
\]
is attained on the ball with the constant load $\rho \equiv 1$ in dimensions $N\ge 2$. 
 
Clearly, if Problem 3 has a positive answer, then the same occurs for problems 1 and 2. Let us point out that at fixed $\Om$, the optimal load $\rho$ is always a bang-bang function. Indeed, in Problem 1, the optimal load equals  $ \sign(u_{\Om, \rho})$ due to \eqref{abl30}. In Problems $2$ and $3$, the optimal $\rho$ is a maximizer of a linear or convex functional (because $u$ in \eqref{as01} depends linearly on $\rho$) taken over the convex set $L^\infty(\Om, [-1,1])$. Consequently the maximizer takes only the values $\pm1$.

\subsection*{\bf Alternative proof of Theorem \ref{th:total} by symmetrization on a single ball}

We would like to mention another way of proving Theorem \ref{th:total}, which resorts to a new Talenti-style comparison principle, and allows to symmetrize directly over a single ball. In this way, we bypass the two-ball problem.

Unfortunately, this comparison principle works properly only in dimension 2 and hence cannot be used to go further than Theorem \ref{th:total}. On the other hand, it could perhaps be generalized in the non-Euclidean setting, allowing one to prove the analogue of Theorem \ref{th:total} for domains in $M=\Sph^2$ and $M=\Hyp^2$. We  emphasize that with this strategy we are able as well to treat the equality case. The comparison principle goes as follows. 
\begin{theorem}[Leylekian \protect{\cite[Section 7.1]{L24thesis}}] \label{thm:talenti-signe}
Suppose $\Omega\subseteq\R^N$ is bounded and open. Let $u\in H_0^2(\Omega)$, assume $\{\nabla u=0\}$ is negligible, and put $f=-\Delta u$. If $v\in H_0^1(\Omega^*)$ solves
\[
\left\{
\begin{array}{rcll}
-\Delta v & = & f^* & \text{in }\Omega^*,\\
v & = & 0 & \text{on }\partial\Omega^*,
\end{array}
\right.
\]
then $v\in H_0^2(\Omega^*)$ and $v$ is radially symmetric and decreasing, hence nonnegative, and for any $1\leq p<\infty$ one has 
\[
\int_{\Omega^*}|v|^p\geq\left[\int_0^{|\Omega_+|}|u^\mystar|+\int_{|\Omega_+|}^{|\Omega|}\left(\frac{(N-2)|\Omega|}{Nz}+1\right)(|\Omega|/z-1)^{(N-2)/N}|u^\mystar(z)| \, dz\right]^p \! |\Omega|^{1-p},
\]
where $u^\mystar$ is the $1$-dimensional decreasing rearrangement of $u$ and $\Omega_+=\{u>0\}$. In particular, if $N=2$ we obtain
\[
\int_{\Omega^*}|v|^p\geq \left(\int_\Omega|u|\right)^p |\Omega|^{1-p} .
\]
Furthermore, if equality holds then $p=1$ and $\Omega$ is a ball up to a negligible set.
\end{theorem}
The main interest of Theorem \ref{thm:talenti-signe} is that it makes no sign assumption on $u$. Thus the theorem can be applied directly to a function without splitting into positive and negative parts.

The proof of this comparison principle can be found in Leylekian \cite[Section~7.1]{L24thesis}. We sketch it below for the reader's convenience.
\begin{proof}
The fact that $v$ belongs to $H_0^2(\Omega^*)$ follows from the usual relation
\[
\int_{\partial\Omega^*}\partial_nv=-\int_{\Omega^*}f^*=-\int_\Omega f=\int_{\partial\Omega}\partial_nu=0 
\]
since $u \in H_0^2(\Omega)$. That $v$ is radially decreasing comes from the maximum principle since $-\Delta \partial_r v=\partial_r f^*\leq 0$ and $\partial_r v=0$ at $r=0, R$, where $R$ is the radius of $\Om^*$.

Now let us prove the comparison result. By classical arguments (see for instance \eqref{eq:Ueq}), we have for all $s\in\R$ that 
\[
\int_{\{u>s\}} f \, dV=\int_{\{u=s\}}|\nabla u| \, dP=\int_{\{u=s\}\cap\Omega}|\nabla u| \, dP,
\]
the last relation coming from the fact that $\nabla u=0$ on $\partial\Omega$. Then, the Cauchy--Schwarz inequality provides that 
\[
P_\Om(\{u>s\})^2\leq \int_{\{u=s\}\cap\Omega}\frac{1}{|\nabla u|}dP\int_{\{u>s\}} f \, dV,
\]
where $P_\Om$ is the relative perimeter with respect to $\Om$. But one can show with standard arguments (see for instance the beginning of the proof of Theorem \ref{talenti}) that $\int_{\{u=s\}\cap\Omega}\frac{1}{|\nabla u|} \, dP=-\mu'(s)$, where $\mu$ is the distribution function of $u$. The Hardy--Littlewood inequality then gives
\begin{equation} \label{HL}
P_\Om(\{u>s\})^2\leq |\mu'(s)|\int_0^{\mu(s)}f^\mystar.
\end{equation}

Now comes the crucial discussion. Since $u=0$ on $\partial\Om$, we need to consider the cases $s>0$ and $s<0$ separately. If $s>0$ then $P_{\Om}(\{u>s\})= P(\{u>s\})$, while if $s<0$ then $P_{\Om}(\{u>s\})= P(\Omega\setminus\{u>s\})$. As a result, the isoperimetric inequality gives
\[
P_\Om(\{u>s\})\geq C_N
\begin{cases}
V(\{u>s\})^{(N-1)/N}=\mu(s)^{(N-1)/N} & \text{if }s>0,\\
[V(\Omega)-V(\{u>s\})]^{(N-1)/N}=[|\Omega|-\mu(s)]^{(N-1)/N} & \text{if }s<0,
\end{cases}
\]
where $C_N$ is the isoperimetric ratio for balls. Combining this inequality with the previous one, we find that
\[
\1_{s>0}+\1_{s<0}\left[\frac{|\Omega|}{\mu(s)}-1\right]^{2(N-1)/N}\leq\frac{|\mu'(s)|}{C_N^2\mu(s)^{\frac{2(N-1)}{N}}}\int_0^{\mu(s)}f^\mystar.
\]
Apply the inequality with $s=u^\mystar(x)$, and recall that $\mu(u^\mystar(x))=x$ and hence $\mu'(u^\mystar(x))=1/{u^\mystar}'(x)$, since we assumed that $\{\nabla u=0\}$ is negligible. Therefore,
\[
[\1_{u^\mystar(x)>0}+g(x)^2\1_{u^\mystar(x)<0}]|{u^\mystar}'(x)|\leq\frac{1}{C_N^2x^{\frac{2N-2}{N}}}\int_0^{x}f^\mystar,
\]
where $g(x)=\left[\frac{|\Omega|}{x}-1\right]^{(N-1)/N}$. Solving explicitly the equation fulfilled by $v$ (see e.g.\ step 4 in the proof of  \cite[Theorem 3.1.1]{Kesavan2006}), one realizes that the right side is exactly $-{v^\mystar}'(x)$. Integrating with respect to $x \in (z,|\Omega|)$, we finally obtain
\[
\int_z^{|\Omega|}[\1_{u^\mystar(x)>0}+g(x)^2\1_{u^\mystar(x)<0}]|{u^\mystar}'(x)| \, dx\leq v^\mystar(z).
\]
Now, put $G(x)=\1_{u^\mystar(x)>0}+g(x)^2\1_{u^\mystar(x)<0}$, and use Jensen's inequality to find
\begin{align*}
\int_0^{|\Omega|}|v^\mystar(z)|^p \, dz 
& \geq |\Omega|^{1-p}\left[\int_0^{|\Omega|} \int_z^{|\Omega|}G(x)|{u^\mystar}'(x)| \, dx dz \right]^p \\
&= |\Omega|^{1-p}\left[\int_0^{| \Omega|}xG(x)|{u^\mystar}'(x)| \, dx\right]^p 
\end{align*}
by Fubini. Since $|{u^\mystar}'(x)|=-{u^\mystar}'(x)$, to deduce the inequality in the theorem it remains only to integrate by parts over the intervals $(0,|\Om_+|)$ and $(|\Om_+|,|\Om|)$, noting that although $G$ is discontinuous at $x=|\Om_+|$, the boundary contribution from that point vanishes because $u^\mystar(|\Om_+|)=0$.

Observe that in case of equality, the rigidity in Jensen's inequality yields that $u$ is constant (which is impossible since $\{\nabla u=0\}$ is negligible), except if $p=1$. In this case, we also have equality in the isoperimetric inequality, which means that almost all the superlevel sets $\{u>s\}$ with $s>0$ are balls, and if $\{u<0\}$ is nonempty then almost all sublevel sets $\{u<s\}$ with $s<0$ are balls too. As a result, we may express $\{u>0\}$ as a countable union of (not necessarily concentric) balls, and similarly for $\{u<0\}$ (if it is nonempty). Since $\{u=0\}$ is negligible and $\Omega$ is open, it follows that $\Omega$ is either a ball or the union of two disjoint balls. We show it cannot be the union of two balls. Note $f$ integrates to zero on each of those balls, since $u \in H^2_0(\Om)$, and so $f$ takes both positive and negative values on each ball. Consequently the Hardy--Littlewood inequality \eqref{HL} is strict for some interval of $s$ values, contradicting that equality holds in the conclusion of the theorem.

\end{proof}

\begin{proof}[Alternative proof of Theorem \ref{th:total} for $N=2$, by using Theorem \ref{thm:talenti-signe}]
The only issue to fix before applying Theorem \ref{thm:talenti-signe} is that $u_{\Omega,\rho}$ might not be analytic. To circumvent this problem, instead of $u_{\Omega,\rho}$ we consider initially an arbitrary function $u\in H_0^2(\Omega)$ such that $\{\nabla u=0\}$ is negligible. In this case, if $v$ is the function vanishing at the boundary such that $-\Delta v=(-\Delta u)^*$ in $\Omega^*$, Theorem \ref{thm:talenti-signe} with $N=2$ and $p=1$ yields that $v\in H_0^2(\Omega^*)$ is nonnegative and
\[
\int_\Omega|\Delta u|^2-2\int_\Omega\rho u\geq \int_{\Omega^*}|\Delta v|^2-2\int_{\Omega^*}v\geq -\MC(\Omega^*,1).
\]

Thus the inequality in the theorem follows provided we show that functions in $H_0^2(\Omega)$ for which $\{\nabla u=0\}$ is negligible are dense in $H_0^2(\Omega)$. 
To do so, for an arbitrary function $u$ in $H_0^2(\Omega)$, we consider $f=\Delta^2 u$ and define $u_k\in H_0^2(\Omega)$ by the relation $\Delta^2 u_k=f+(1/k)\1_{\{f=0\}}$. In this way, $u_k$ converges to $u$ as $k$ goes to infinity. Further, $\Delta^2 u_k\neq 0$ on $\Omega$, which prevents $\nabla u_k$ from vanishing on a nonnegligible subset.

Finally, if $\MC(\Omega,\rho)=\MC(\Omega^*,1)$ then $\Omega$ is a ball up to a negligible set, by the equality case in Theorem \ref{thm:talenti-signe}. Then, as previously, we have that $\Omega\subseteq \Omega^*$ pointwise since the interior of $\overline{\Omega^*}$ is $\Omega^*$. As a result, $u_{\Omega,\rho}$ is in $H_0^2(\Omega^*)$ and it optimizes $\MC(\Omega^*,1)$. This fact yields $u_{\Omega,\rho}=u_{\Omega^*,1}$ $H^2$-q.e., and since $u_{\Omega^*,1}>0$, it must be that $\Omega^*\subseteq \Omega$ up to sets of  null $H^2$-capacity. Lastly, $(1-\rho)u_{\Omega^*,1}$ is a nonnegative function over $\Omega^*$, but its integral is $\MC(\Omega^*,1)-\MC(\Omega,\rho)=0$, and so it must vanish a.e. From the positivity of $u_{\Omega^*,1}$ we conclude that $\rho=1$ a.e.
\end{proof}

\section{\bf Possible alternative strategies}
\label{sec:alternative}

Let us comment on possible different strategies of proof, focusing on the simplest situation of Euclidean space with no compression and constant load $\rho \equiv 1$. 

\subsection*{Overdetermined boundary condition} A possibility for proving Theorem \ref{as10general} is to formulate an overdetermined problem, similar to Serrin's. For the buckling load conjecture, this approach was implemented by Weinberger and Willms; see for instance \cite{AB03}. In order to obtain a full proof, such a strategy requires a priori knowledge of the existence of an optimal domain that is smooth enough and has connected boundary. Assuming this information can be obtained, then  the Hadamard formula (based on a shape derivative argument) together with the smoothness of $u_\Om$ and  the connectedness of the boundary, lead to the overdetermined boundary condition 
\[
\Delta u_\Om = \text{constant} \text{ on } \partial \Om.
\]
For the compliance functional, this overdetermined boundary problem has been solved by Bennett \cite{Be86} under  regularity hypotheses of  the boundary. See also Payne and Schaefer \cite{PaSc94}, Barkatou \cite{Ba08}, and Dalmasso \cite{Da90}.  However, for now the proof remains incomplete, because the regularity of the free boundary is not fine enough and the boundary connectedness of an optimal set is a difficult technical problem that is completely open. 
 
\subsection*{Rayleigh quotient approach}
An alternative approach to the mean deflection problem under constant load is to work with the Rayleigh quotient. This approach is described in Leylekian's PhD thesis \cite[Theorem 5.18]{L24thesis}, where he considers the compliance problem \eqref{as03.1} in the Euclidean setting with no compression and constant load $\rho \equiv 1$ (that is, the Euclidean case of our Theorem \ref{as10general}). Interpreting the inverse of the compliance as a Rayleigh quotient of a semilinear eigenvalue, one can follow the Talenti/Ashbaugh--Benguria framework for the first clamped plate eigenvalue  and build  a two-ball auxiliary problem for the Rayleigh quotient. The solution of the auxiliary problem can be computed exactly and from that explicit formula one can show the optimal configuration is a single ball. Note that this auxiliary problem treats the nodal regions $\Omega_+$ and $\Omega_-$ symmetrically, and consequently differs from our functional \eqref{eq:asymmetricgeneral} in the proof of Theorem \ref{as10general}. This difference raises hope that some additional helpful information could be extracted. However, that does not seem to be the case.

\section*{Acknowledgments}
Richard Laugesen's research was supported by grants from the Simons Foundation (\#964018) and the National Science Foundation 
({\#}2246537).  Rom\'eo Leylekian was partially supported by the Funda\c c\~{a}o para a Ci\^{e}ncia e a Tecnologia, I.P. (Portugal), through project {\small doi.org/10.54499/UIDB/00208/2020}. We are grateful to Almut Burchard for guidance on distribution functions.

\appendix

\section{\bf Talenti-style elliptic rearrangement on constant curvature spaces}

The rearrangement result we need for proving Theorem \ref{as10general} does not seem to exist in the literature in quite the needed form, and so here we state and prove it in a form convenient for our purposes. 

Recall that $M$ denotes either Euclidean space $\RN$, the round unit sphere $\Sph^N$, or $N$-dimensional hyperbolic space $\HN$ with sectional curvature $-1$. Assume $N \geq 1$ for the Euclidean space, and $N \geq 2$  for the sphere and hyperbolic space. In each case, the metric is denoted by $\mathfrak{g}$, the gradient is $\nabla=\nabla_{\! \mathfrak{g}}$, the Laplace--Beltrami operator is $\Delta_\mathfrak{g}$, and the volume element is $dV=dV_\mathfrak{g}$. The analyst's Laplacian is $\Delta=-\Delta_\mathfrak{g}$. Note $|\nabla u|$ means $\mathfrak{g}(\nabla_{\! \mathfrak{g}} u,\nabla_{\! \mathfrak{g}} u)^{1/2}$.
\begin{theorem} \label{talenti}
Let $\Omega \subset M$ be an open set with finite volume (and in the spherical case assume $1 \notin H^2_0(\Om)$) and take $B \subset M$ to be an open geodesic ball with the same volume. Suppose $u \in H^1_0 \cap W^{2,1}(\Omega)$ is $C^2$-smooth and the critical set $\{ |\nabla u|=0 \}$ has zero volume, and $f \in H^1_0 \cap W^{2,1}(B)$ is $C^2$-smooth away from the center of $B$, and that they solve the Poisson equations  
\[
\Delta u = U \quad \text{ on } \Omega , \qquad \Delta f = F \quad \text{ on } B ,
\]
where the symmetric decreasing rearrangement of $-U$ is $-F$. 

If $u \geq 0$ then the symmetric decreasing rearrangement of $u$ is bounded by the radial function $f$, that is $0 \leq u^* \leq f$, and so 
\begin{equation} \label{eq:Tal1}
\int_\Omega u^p \, dV \leq \int_B f^p \, dV , \qquad p \in [1,\infty) , 
\end{equation}
and also 
\begin{equation} \label{eq:Tal2}
\int_\Omega |\nabla u|^q \, dV \leq \int_B |\nabla f|^q \, dV , \qquad  q \in (0,2].
\end{equation}
Further, $f$ is strictly radially decreasing and positive, with $\partial f/\partial r < 0$ and $f>0$ on $B$. Finally, if equality holds in \eqref{eq:Tal1} or \eqref{eq:Tal2} then $\Omega$ is a geodesic ball up to a set of measure $0$. 
\end{theorem}
The symmetric decreasing rearrangement of a function $h$ on $\Omega$ is a radial function $h^*$ on $B$ whose superlevel sets have the same volume as those of $h$: 
\[
V(\{ x \in \Omega : h(x)>a \})=V(\{ x \in B : h^*(x)>a \}) , \quad a \in \R .
\]

\subsubsection*{Notes on the literature.} Elliptic comparison theory goes back to P\'{o}lya and Szeg\H{o} \cite[Note F]{PolyaSzego1951}, who in the Euclidean case proved comparisons for the $L^2$ norm of $u$ and $\nabla u$. The theory is most closely associated with Talenti \cite{Tal76}, who established that $u^* \leq f$ and $\lVert \nabla u \rVert_q \leq \lVert \nabla f \rVert_q$, and also treated more general elliptic operators. Talenti assumed $\Delta u \leq 0$, but later recognized in \cite{Tal81} that it was more useful to assume instead $u \geq 0$, which is the version of the result presented in the elegant text by Kesavan \cite[Theorem 3.1.1]{Kesavan2006}. 

A generalization to manifolds satisfying a Ricci curvature bound was given by Colladay, Langford and McDonald \cite[Theorem 3.4]{ColladayLangfordMcDonald18}. See also Krist\'{a}ly  \cite[Theorem 3.1]{Kristaly2020}, \cite[Theorem 4.1]{Kristaly2022}. Those authors do not state a comparison result on $\nabla u$. Note Krist\'{a}ly used a slightly different rearrangement than we do here, in which (borrowing terminology from earlier in this paper) the Laplacian of $u_\Omega$ is rearranged on all of $\Omega$ before being split up for use in the comparisons on $\Omega_+$ and $\Omega_-$. That yields a slightly weaker result than first splitting $u$ on $\Omega_\pm$ and then rearranging the Laplacian on each piece, which is how we proceed in the current paper. 

\begin{proof}[Proof of Theorem \ref{talenti}]
We follow Talenti's method except that, as he surely had in mind when developing the theory, by restricting to functions whose critical set has volume zero we are able to work more directly with level sets. Further, we avoid explicit expressions for $f$ in terms of $F$, so that the proof flows easily in all three geometries and highlights the analogies between $u$ and $f$. 

A well known fact needed for the proof is that the distribution function $\mu(a)=V(u>a)$ is absolutely continuous. For readers' convenience, we include the short proof. For each $\e>0$, the coarea formula  yields that 
\[
\int_{\{ u>a \}} \frac{|\nabla u|}{\e+|\nabla u|} \, dV = \int_a^\infty \int_{\{ u=a \}} \frac{1}{\e+|\nabla u|} \, dP da ,
\]
where $dP$ is the surface area (perimeter) element. Letting $\e \to 0$ and applying monotone convergence, one obtains
\[
\int_{\{ u>a \}} 1_{\{ |\nabla u| \neq 0 \}} \, dV = \int_a^\infty \int_{\{ u=a \}} \frac{1}{|\nabla u|} \, dP da . 
\]
The set where $|\nabla u|=0$ has zero volume by hypothesis and so the integral on the left side evaluates to $\mu(a)$, giving the desired absolute continuity with derivative 
\[
- \mu^\prime(a) = \int_{\{u=a\}} \frac{1}{|\nabla u|} \, dP > 0
\]
for almost every $a$ in the range of $u$. 

Continuity of the distribution function implies that each level set $\{u=a\}$ has measure zero. 

\smallskip
\textsc{Step 1 ($f$ is strictly radially decreasing)}. Take $u$ and $f$ as in the theorem, so that in particular $u \geq 0$. 

Since $-F$ is equimeasurable with $-U$, we find  
\begin{equation} \label{eq:Fint}
\int_B F \, dV = \int_\Omega U \, dV = \int_\Omega \Delta u \, dV = \int_{\partial \Omega} \frac{\partial u}{\partial n} \, dP \leq 0 ,
\end{equation}
using that $u$ is nonnegative and satisfies the Dirichlet boundary condition. Strictly speaking, the application of the divergence theorem here requires some additional regularity of $\partial \Omega$. Talenti \cite[p.{\,}170]{Tal79} showed how to circumvent that concern through a good choice of trial function, getting that  
\begin{equation} \label{eq:Ueq}
\int_{\{ u>a \}} U \, dV = \frac{d\ }{da} \int_{\{ u>a \}} |\nabla u|^2 \, dV = - \int_{\{ u=a \}} |\nabla u| \, dP \leq 0
\end{equation}
for almost every $a>0$ (relying on the coarea formula for the second equality), after which letting $a \to 0$ yields $\int_\Omega U \, dV \leq 0$ as needed for \eqref{eq:Fint}. 

If $F(0) \geq 0$ then $F \geq 0$ since $F$ is radially increasing, and so $F \equiv 0$ by \eqref{eq:Fint} which means $U \equiv 0$ too, and thus $u$ vanishes identically, contradicting the hypotheses. Hence $F(0)<0$, which implies by \eqref{eq:Fint} (once more using that $F$ is radially increasing) that $\int_{B(r)} F \, dV < 0$ for all $r$ smaller than radius of $B$. Thus $\partial f/\partial r < 0$ and so $f$ is strictly radially decreasing, and $f>0$ on $B$. 

Further, each level set $\{f=b\}$ is a geodesic sphere and so has measure zero.  

\smallskip
\textsc{Step 2 (matching superlevel volumes)}. For each $a \in [0,\lVert u \rVert_\infty]$, a unique number $b \in [0,\lVert f \rVert_\infty]$ exists such that 
\[
V(u>a) = V(f>b) ,
\]
as follows. Indeed, the distribution function $\mu(a)=V(u>a)$ is an absolutely continuous function that is strictly decreasing (since $u$ is continuous) with $\mu(\lVert u \rVert_\infty)=0$ and $\mu(0)=V(\Omega)$, and similarly $\nu(b)=V(f>b)$ is strictly decreasing with $\nu(\lVert f \rVert_\infty)=0$ and $\nu(0)=V(B)=V(\Omega)$. Hence the desired $b$ value is 
\[
b(a)=\nu^{-1}(\mu(a)) .
\]
The map $a \mapsto b(a)$ is strictly increasing, and in fact $b(a)$ is absolutely continuous because $\nu^{-1}$ is $C^2$-smooth on $(0,V(B))$ and hence Lipschitz there, while $\mu$ is absolutely continuous. 

\smallskip
\textsc{Step 3 (comparing gradients on level sets)}. Suppose $a$ is a regular value for $u$, so that $|\nabla u| \neq 0$ on the level set $\{ u=a \}$. The integral of $|\nabla u|$ over that level set is less than or equal to the integral of $|\nabla f|$ over the corresponding level set of $f$: 
\begin{align}
0 < \int_{\{u=a\}} |\nabla u| \, dP 
& = \int_{\{u>a\}} (-\Delta u) \, dP \qquad \text{by \eqref{eq:Ueq}} \notag \\
& \leq \int_{\{f>b\}} (-\Delta f) \, dP \label{talentigrad} \\
& = \int_{\{f=b\}} |\nabla f| \, dP \label{talentigrad2} 
\end{align}
where inequality \eqref{talentigrad} uses that the symmetric decreasing rearrangement of $-\Delta u = -U$ is the function $-F=-\Delta f$ and the sets $\{u>a\}$ and $\{f>b\}$ over which the two functions are integrated have the same volume by our choice of $b$; the final equality uses the divergence theorem, noting $-\partial f/\partial n = |\nabla f|$ because $f$ is radially decreasing. 

\smallskip
\textsc{Step 4 (comparing reciprocal gradients)}. Write $P$ for perimeter in the sense of De Giorgi, so that $P(u>a)=\int_{\{u=a\}} \, dP$ for almost every $a$, by \cite[p.{\,}712]{Tal76}. Then Cauchy--Schwarz on the level set implies that 
\begin{align*}
\frac{1}{P(u>a)^2} \int_{\{u=a\}} \frac{1}{|\nabla u|} \, dP 
& \geq \frac{1}{\int_{\{u=a\}} |\nabla u| \, dP} \\
& \geq \frac{1}{\int_{\{ f=b\}} |\nabla f| \, dP} && \text{by \eqref{talentigrad2}} \\
& = \frac{1}{P(f>b)^2} \int_{\{f=b\}} \frac{1}{|\nabla f|} \, dP 
\end{align*}
where we used that $f$ is radially decreasing and so $|\nabla f|$ is constant on the sphere $\{ f=b \}$; note that the sphere equals the boundary of the geodesic ball $\{ f>b \}$. 

The isoperimetric inequality holds for subdomains of Euclidean, spherical and hyperbolic space (see \cite[Theorem 14.3.1]{BZ88}), and so since the superlevel sets of $u$ and $f$ have the same volume, their perimeters satisfy $P(u>a) \geq P(f>b)$. Hence the displayed inequality implies 
\[
\int_{\{u=a\}} \frac{1}{|\nabla u|} \, dP \geq \int_{\{f=b\}} \frac{1}{|\nabla f|} \, dP 
\]
(where we note that the right side of the inequality is positive), or equivalently 
\begin{equation}
- \mu^\prime(a) \geq - \nu^\prime(b) .
\label{talentireciprocal} 
\end{equation}
If equality holds in \eqref{talentireciprocal} then equality must hold in the isoperimetric inequality, and so (again by \cite[Theorem 14.3.1]{BZ88}) the superlevel set $\{u >a\}$ is a geodesic ball up to a set of measure zero. 

\smallskip
\textsc{Step 5 ($b$ grows faster than $a$)}. For almost every value $a$, 
\begin{align}
0 < - \mu^\prime(a) 
& = - \frac{d\ }{da} \nu(b) \qquad \text{since $\mu(a)=\nu(b)$} \notag \\
& = - \nu^\prime(b) \frac{db}{da} \label{talentibprime} \\
& \leq - \mu^\prime(a) b^\prime(a) \qquad \text{by \eqref{talentireciprocal},} \notag
\end{align}
and so $b^\prime(a) \geq 1$ a.e. If equality holds, $b^\prime(a) = 1$, then equality holds in \eqref{talentireciprocal} and so $\{u >a\}$ is a ball up to a set of measure zero.  

\smallskip
\textsc{Step 6 ($u^*$ is less than $f$)}. Since $b(0) = 0$ and $b(a)$ is absolutely continuous, Step 5 implies $b(a) \geq a$ for $a \in [0,\lVert u \rVert_\infty]$ and hence 
\[
V(u^* > b) \leq V(u^*>a) = V(u>a) = V(f>b) 
\]
for each $b$. Therefore the radially decreasing functions $u^*$ and $f$ satisfy $u^* \leq f$ on $B$. It follows by the classic technique of majorization that $\int_\Omega \Phi(u) \, dV \leq \int_B \Phi(f) \, dV$ for every convex increasing $\Phi : [0,\infty) \to \R$. In particular, $\int_\Omega u^p \, dV \leq \int_B f^p \, dV$ for each $p \geq 1$. If equality occurs for some $p$ then by the layer cake representation of the $L^p$ norm in terms of the distribution function, we deduce $b^\prime(a)=1$ for almost every level $a$. Hence $\{u>a\}$ is a ball up to a set of measure $0$, for each such $a$. 

\smallskip
\textsc{Step 7 (comparing gradient norms)}. 
When $0 < q \leq 2$,  
\[
\begin{split}
\int_{\{u=a\}} |\nabla u|^{q-1} \, dP 
& = \int_{\{u=a\}} |\nabla u|^{q/2} \, |\nabla u|^{-(2-q)/2} \, dP \\
& \leq \left( \int_{\{u=a\}} |\nabla u| \, dP \right)^{\!\! q/2} \left( \int_{\{u=a\}} |\nabla u|^{-1} \, dP \right)^{\!\! (2-q)/2} 
\end{split}
\]
by H\"{o}lder. Hence by \eqref{talentigrad2} and \eqref{talentibprime}, 
\begin{align}
\int_{\{u=a\}} |\nabla u|^{q-1} \, dP 
& \leq  \left( \int_{\{f=b\}} |\nabla f| \, dP \right)^{\!\! q/2} \left( \int_{\{f=b\}} |\nabla f|^{-1} \, dP \, b^\prime(a) \right)^{\!\! (2-q)/2} \notag \\
& \leq b^\prime(a) \int_{\{f=b\}} |\nabla f|^{q-1} \, dP \label{eq:bprime}
\end{align}
since $|\nabla f|$ is constant on $\{ f=b \}$ and $b^\prime(a) \geq 1$. Integrating the preceding inequality with respect to $a$, we find 
\begin{align*}
\int_\Omega |\nabla u|^q \, dV 
& = \int_0^\infty \int_{\{u=a\}} |\nabla u|^{q-1} \, dP da \\
& \leq \int_0^\infty \int_{\{f=b\}} |\nabla f|^{q-1} \, dP \, \frac{db}{da} da = \int_B |\nabla f|^q \, dV ,
\end{align*}
which completes the proof of the inequality in the theorem. If equality occurs for some $q$, then for almost every level $a$ one has equality in \eqref{eq:bprime} and hence $b'(a)=1$, implying that $\{u>a\}$ is a ball up to a set of measure $0$. 

\smallskip
\textsc{Step 8 (equality case)}. 
If equality holds in \eqref{eq:Tal1} or \eqref{eq:Tal2}, then as shown in Steps 6 and 7, almost every superlevel set $\{ u>a \}$ is a geodesic ball up to a set of measure $0$. By taking a sequence of such $a$-values decreasing to $0$, we express $\Omega$ as a countable union of balls (not necessarily concentric), minus a set of measure $0$. This expanding union of balls is itself a ball. Thus $\Omega$ is a geodesic ball up to a set of measure $0$. 
\end{proof}
%


\end{document}